\documentclass[a4paper,11pt]{article}
\usepackage{amsfonts}
\usepackage[square, numbers]{natbib}

\usepackage[left=0.8in,top=0.8in,right=0.8in,bottom=1.0in]{geometry}
\usepackage{graphics,epsfig}
\usepackage{color}
\usepackage{amsmath,amsthm,amssymb,amsfonts, mathrsfs, epsfig,latexsym,graphicx, enumerate}

\usepackage{macro}

\usepackage{tikz}
\usetikzlibrary{calc, quotes, angles}

\newtheorem{theorem}{Theorem}[section]
\newtheorem{Remark}{Remark}[section]
\theoremstyle{definition}
\newtheorem{dfn}{Definition}[section]
\theoremstyle{remark}

\theoremstyle{lemma}
\newtheorem{lemma}[theorem]{Lemma}
\newtheorem{cor}[theorem]{Corollary}
\newtheorem{proposition}[theorem]{Proposition}
\numberwithin{equation}{section}

\def\O{\Omega}
\def\o{{\omega}}
\def\bA{{\bf A}}
\def\bB{{\bf  B}}
\def\F{{\mathcal F}}
\def\l{\lambda}
\def\p{\partial}
\def\t{\times}
\def\e{\epsilon}
\def\f{{\mathbf f}}
\def\bg{\mathbf g}
\def\bu{\mathbf u}
\def\ue{{\mathbf u}^{\epsilon}}
\def\ve{{\mathbf v}^{\epsilon}}
\def\we{{\mathbf w}^{\epsilon}}
\def\ye{{\mathbf y}^{\epsilon}}
\def\ze{{\mathbf z}^{\epsilon}}
\def\bv{\mathbf v}
\def\bw{\mathbf w}
\def\bz{\mathbf z}
\def\by{\mathbf y}
\def\bE{\mathbb{E}}
\def\n{{\mathbf n}}
\def\a{\alpha}
\def\d{\delta}
\def\k{\kappa}
\def\s{\sigma}
\def\th{\theta}
\def\E{\mathbb{E}}
\def\R{{\mathbb{R}}}
\def\N{{\mathbb N}}
\def\Rn{\R^n}
\def\ms{\medskip}
\def\mni{\medskip\noindent}
\def\ss{\smallskip}
\def\sni{\smallskip\noindent}
\def\ni{\noindent}
\def\pr{^\prime}
\def\dsum{\displaystyle \sum}
\def\dfrac{\displaystyle \frac}
\def\dpr{^{\prime\prime}}
\def\mybaselineskipA{\baselineskip 0.25in}
\def\mybaselineskipB{\baselineskip=10pt}
\providecommand{\norm}[1]{\lVert#1\rVert}

\begin{document}
\title{ Identification and existence of Boltzmann  processes}

\author{ B. R\"udiger\footnote{Bergische Universit\"at Wuppertal Fakult\"at 4 -Fachgruppe Mathematik-Informatik,
Gauss str. 20, 42097 Wuppertal, Germany. Email: ruediger@uni-wuppertal.de Orcid:0000-0002-4363-2812; corresponding author}
 and P. Sundar \footnote{Department of Mathematics, Louisiana State University, Baton Rouge, La 70803, USA. Email: psundar@lsu.edu Orcid:0009-0000-0410-9079}}
\maketitle
\begin{abstract}
 \noindent 
 The stochastic differential equation of McKean-Vlasov type is identified such that  the Fokker-Planck equation associated to it is the Boltzmann equation. Hence, we call its solutions as Boltzmann processes. They describe the dynamics (in position and velocity) of particles expanding in vacuum in accordance with  the Boltzmann equation. Given a solution $f:=$  $\{f(t,x,v\}_{0 \le t \le T} $ of the Boltzmann equation, the existence of solutions to the McKean-Vlasov SDE is established for the non-cutoff hard sphere case. 
  \\

\end{abstract}
{\bf AMS Subject Classification:} {35Q20; 60H20; 60H30}\\
{\bf Keywords:} {The Boltzmann equation, Poisson random measure, stochastic differential equation, relative entropy}

\section{Introduction} \label{Introduction}
The Boltzmann equation describes the time evolution of the density function in a position-velocity space for a classical particle (molecule) 
under the influence of  other particles in a diluted (or rarified) Boltzmann gas \cite{Bo} evolving in vacuum. It forms the basis for 
the kinetic theory of gases, see, for e.g. the article by C. Cercignani \cite{Ce4}. 

\par
The aim of this work is to identify a stochastic process for which the associated  Fokker-Planck equation coincides with the Boltzmann equation in weak form. We call such processes as Boltzmann processes. 

Let $f$ be the particle  density function, which depends on time $t \ge 0$, the space variable $x \in \R^3$, and the velocity variable $z  \in \R^3$ of the point particle.  
 The Boltzmann equation has the general form 
\begin{equation}
\frac{ \partial f}{\partial t} (t, x, z)+ z \cdot \nabla_x f(t, x, z) = Q(f, f)(t,x,z), \label{eqn1.1}
\end{equation}
where $Q$ is a certain quadratic form in $f$, called collision (integral) operator. Throughout the paper $a \cdot b$ or $(a, b)$ is used to denote the scalar product of two vectors. 

\par 
\noindent Set $\Xi := (0, \pi] \times [0, 2\pi)$. Then $Q$ can be written in spherical co\"ordinates as 
\begin{equation}
Q(f, f) (t,x,z)= \int_{\R^3\times \Xi} \{f(t, x, z^\star) f(t, x, v^\star)-f(t, x, z) f(t, x, v)\}B(z, dv, d\theta)d\phi. \label{eqn1.2}
\end{equation}
 Each $v \in \R^3$ in (\ref{eqn1.2}) denotes the velocity of an incoming  particle which may hit, at the fixed location $x\in \mathbb{R}^3$,  particles whose velocity is fixed as 
$z\in \mathbb{R}^3$. 
Let $z^\star \in \R^3$ and $v^\star \in \R^3$ denote the resulting outgoing velocities corresponding to the incoming velocities $z$ and  $v$ respectively. 
$\theta$ $\in (0,\pi]$ denotes the azimuthal or colatitude angle of the deflected velocity, $v^{\star}$ (see \cite{Ta}).  Having determined $\theta$, 
the longitude angle $\phi \in [0, 2\pi)$ measures in polar coordinates,  the location of  $v^*$, and hence  that of $z^\star$, as explained below. 
 
\noindent In the Boltzmann model as the collisions are assumed to be elastic,  conservation of kinetic energy as well as momentum of the molecules holds, i.e. 
 considering particles of mass $m =1$, the following equalities hold:
\begin{equation}
\left\{
\begin{aligned}
z^\star+v^\star =z+v\\
|z^\star|^2+|v^\star|^2 =|z|^2+|v|^2 
\end{aligned}
\right. \label {eqn1.3}
\end{equation}
In fact,   
\begin{equation} 
\left\{
\begin{aligned}
z^\star=z+ ({\bf n},v-z) {\bf n}\\
v^\star=v-({\bf n},v-z){\bf n}
\end{aligned}
\right. \label {eqn1.4}
\end{equation}
   
\noindent where 
\begin {equation}\label {eqn1.5}
{\bf n}=\frac{z^\star-z}{|z^\star-z|}
\end {equation}
 where $(\cdot, \cdot)$ denotes the scalar product, and $| \cdot |$, the Euclidean norm in $\mathbb{R}^3$.

\noindent  \begin{Remark} \label {Remoutcomevel} The Jacobian of the transformation (\ref {eqn1.4}) has determinant 1 and  $(z^\star)^\star=z$ since the collision dynamics are reversible.\end {Remark}

\noindent The outgoing velocity  $z^*$ is then uniquely determined in terms of the colatitude angle $\theta\in (0,\pi]$ measured from the center, and longitude angle $\phi \in  [0,2\pi)$ of  the deflection  vector ${\bf n}$ in 
a sphere with northpole $z$ and southpole $v$ centered at $\frac{z+v}{2}$ (and with radius determined by the conserved kinetic energy) 
which are used in equation (\ref{eqn1.1}) and \eqref{eqn1.2} (see  e.g. \cite{Bo}, \cite{Ce}, \cite{Ce4}, \cite {Ta2} for futher details). It follows 

\begin{equation} (v-z,{\bf n})= |v-z|\cos(\frac{\pi}{2}- 
\frac{\th}{2})= |v-z| \sin(\frac{\th}{2}), \label{eqscalar}\end{equation}
where $\theta$ is the angle between $z-v$ and $ z^\star-v^\star$, and $\frac{\pi}{2}- \frac{\theta}{2}$ is the angle between $\bf n$ and $v-z$. In spherical coordinates we obtain
\begin{equation} (v-z,{\bf n})d{\bf n}=|v-z| \sin(\frac{\th}{2})\cos(\frac{\th}{2})d\th d\phi\label{eqcross}\end{equation}
\par 
 Figure 1 below helps to visualize the deflection vector and the angles mentioned above.
\begin{figure}
\begin{center}
\begin{tikzpicture}
  [
    scale=5,
    >=stealth,
    point/.style = {draw, circle,  fill = black, inner sep = 1.3pt},
    dot/.style   = {draw, circle,  fill = black, inner sep = .2pt},
  ]

  \def \rad{1}
  \node (origin) at (0,0) 
  [point, label = {below right:$\frac{v+z}{2}$}]{};;
  \draw (origin) circle (\rad);

  \node (n1) at +(90:\rad) [point, label = above:$z$] {};
  \node (n2) at +(-90:\rad) [point, label = below:$v$] {};
  
  \node (n3) at +(25:\rad)  {};
  \node (n7) at +(-25:\rad)  {};
  \node (n6) at (0.5,0.3) [point, label = {below right:$z^{\star}$}] {};
  
  \node (n4) at (-0.5, -0.3) [point, label = {below:$v^{\star}$}] {};
  
  \node (n5) at (n3-|0,1) [point] {};
  
  \draw[thick, ->, red] (n1) -- node (a) [label = {right:${\bf n} = \frac{z^{\star} - z}{|z^{\star} - z|}$}] {} (n6);
    
  \draw[->, green] (n5) -- (n6);
  \draw[->, blue] (n5) -- (n3);
  \draw[->, blue] (n1) -- (n5);
  \draw 
  pic["$\phi$", draw=orange, <->, angle eccentricity=0.8, angle radius=1.8cm] {angle= n6--n5--n3};
  
  \draw 
  pic["$\frac{\pi - \theta}{2}$", draw=orange, <->, angle eccentricity=0.8, angle radius=1.8cm] {angle= n5--n1--n6};
    
  \draw[->] (origin) -- (n6);
  \draw[->] (origin) -- (n5);
  \draw
   pic["$\theta$", draw=orange, <->, angle eccentricity=0.5, angle radius=1.0cm] {angle= n3--origin--n1};
   
  \draw[dashed] (n3-|0,0) ellipse (0.9 and 0.15);
  \draw[dashed] (n7-|0,0) ellipse (0.9 and 0.15);
  \draw[dashed] (origin) ellipse (1 and 0.15);
  
  \draw[dashed] (-1,0) to (1,0);

\end{tikzpicture}
\end{center}
\caption{Parameterization of collisions}
\end{figure}
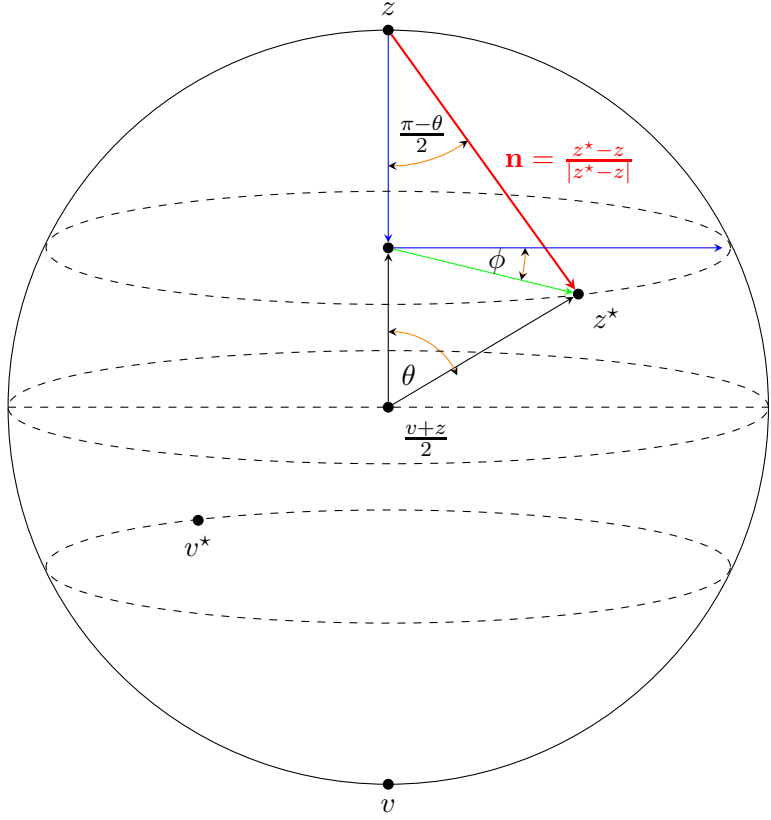

\noindent  The collision measure $B(z, dv, d\theta)$ appearing in \eqref{eqn1.2} has then the form
\begin{equation}\label{scattering measure Boltzmann}
B(z, dv, d\theta)=|v-z| \sin(\frac{\th}{2})\cos(\frac{\th}{2})d\th dv.
\end{equation}
\eqref{eqn1.1} with \eqref{scattering measure Boltzmann} is then the  ''Boltzmann equation with hard spheres'',  proposed by Ludwig Eduard Boltzmann.
In further modifications of the Boltzmann  model the collision measure $B(z, dv, d\theta)$ is a $\sigma$-finite positive measure defined on the Borel $\sigma$-field ${\mathcal B}(\R^3) \otimes {\mathcal B}((0, \pi])$, depending 
 (Lebesgue) measurably on $z \in \R^3$. The form of $B$ depends on the version of Boltzmann equation one has in mind. 
In general one sets  (see e.g. \cite{Vil}) 
\begin{equation}
B(z, dv, d\theta) := \sigma (|v-z|) dv Q(d\theta) \label{eqn1.8}
\end{equation}
where $\sigma$ is some positive function on $\R_0^+$$:=\{t\in \R:t \geq 0\}$, and $Q$ is a $\sigma$-finite measure on ${\mathcal B}((0, \pi])$. 
When $Q(d\theta)$ is taken to be integrable, one speaks of a cut-off function.   
In general the models considered involve  (see e.g. \cite{Vil})
\begin{enumerate}
\item[i)] $\sigma(|z-v|)=|z-v|^\gamma$ for $\gamma \in (-1,1]$
\item [ii)] $Q(d\theta)=\theta^{-1-\nu}d\theta$ for $0<\nu<1$
\end{enumerate}
Note that 
\begin{equation}\label{sigmafinitemeasure}
\int_0^\pi \theta Q(d\theta)< \infty.
\end{equation}
In the literature the molecules following the dynamics of the Boltzmann equation \eqref{eqn1.1} with \eqref{eqn1.2} are called 
\begin{itemize}
\item ''molecules with soft potential'',  if $-1<\gamma<0$ and $\nu>1/2$,
\item  ''molecules with hard potential'',  if $0<\gamma<1$ and $\nu<1/2$,
\item ''Maxwell molecules'',  if   $\gamma=0$ and  ii) holds.
\end{itemize}
We remark that in Section 4 when we consider the Boltzmann equation with hard spheres, then  $\gamma=1$. In general, $Q(d\theta)$ is taken to be a $\sigma$-finite measure.\\

     \noindent  From now on, we will assume that $B(z, dv, d\theta)$ is as in \eqref{eqn1.8} and satisfies  the following hypothesis A, which include all the above mentioned models.
            
   \noindent
   {\bf Hypotheses A:}
   
   \begin{itemize}
   \item[\bf{A1.}] The measure $Q$ is finite outside any neighbourhood of $0$, and for all $\epsilon > 0$, it satisfies 
   \[
    \int_0^{\epsilon} \theta Q(d\theta) < \infty.
   \]
   \item[\bf{A2.}] $\sigma: \R_0^+ \to \R_0^+$ (entering \eqref{eqn1.8}) is given by $\sigma(z):=c |z|^\gamma$, with $c>0$, $\gamma\in (-1,1]$ fixed.
   \end{itemize}

   \par There are many useful consequences of {\bf A1}. Let us set  
   
   \begin{equation} \alpha (z, v, \theta, \phi) := ({\bf n}, (v-z)){\bf n} \label{alpha} \end{equation}
   with $\frac{\pi}{2} - \frac{\th}{2}$ as the angle between the vectors $(v-z)$ and $\bf{n}$. Let us use the notation  
   \begin{equation*}
   \Xi:= (0,\pi]\times [0,2\pi)\,. 
   \end{equation*} 
     We define 
    \begin{equation}
    \hat{\alpha}(z,v, \theta, \phi) := \alpha(z,v, \theta, \phi)\sigma(|z-v|) \label{alphahat}
    \end{equation}
    Condition {\bf A1} implies that there exists a constant $C$ such that the following estimate holds.
      \begin{equation}
   \int_{\Xi}  |\hat{\alpha}(z,v, \theta,\phi)| Q(d\theta) d\phi \leq C|z-v|^{1+\gamma},\label{growth-L}
   \end{equation}
   and hence
   \begin{equation}
   \int_{\Xi}  |\hat{\alpha}(z,v, \theta,\phi)| Q(d\theta) d\phi \leq C(|z|^{1 +\gamma}+|v|^{1+\gamma}).\label{growth}
   \end{equation}


Moreover   
\begin{align}    \label{PrameterTrans}  
\alpha(z,v, \theta,\phi)&= \frac{1-\cos(\theta)}{2}(v-z)+\frac{\sin(\theta)}{2}\Gamma(v-z, \phi)\notag \\&=\sin^2(\frac{\theta}{2})(v-z)+\frac{\sin(\theta)}{2}\Gamma(v-z, \phi)
   \end{align}
 for all $\phi $$\in [0,2\pi)$,  where 
   \begin{equation*}\Gamma(v-z, \phi)=I(v-z)\cos(\phi) +J(v-z)\sin(\phi)  \end{equation*} and  $\frac{v-z}{|v-z|}$, $\frac{I(v-z)}{|v-z|}$, $\frac{J(v-z)}{|v-z|}$ form an orthogonal basis. It follows in particular  that 
   \begin{equation}\label{symmGamma}
   \int_0^{2\pi} \Gamma(v-z, \phi)d\phi=0.
   \end{equation}

  In order to study solutions to the Boltzmann equation, one is naturally led to
study continuity properties of $(z-v,{\bf n}){\bf n}$ in $z,v$ for fixed $\theta,\phi$. However, it was already pointed out by Tanaka that 
$(z,v) \longmapsto (z-v,{\bf n}){\bf n}$ cannot be smooth.
To overcome this problem Tanaka introduced in Lemma 3.1 of \cite{Ta} another transformation of parameters, which describes a rotation around the longitude angle, is bijective and  has Jacobian 1. As a consequence of this transformation $\phi_0$, he proved following Lemma \ref{rotationTa} (see also   Lemma 2.6 in \cite{Ta}).

\begin{lemma}\label{rotationTa}\cite{Ta} (cf. \cite{FM01} or  Section 3 in \cite{FM})
There exists   a measurable function $\phi_0:\,\mathbb{R}^{12} \times (0, 2\pi]\,\to\, (0,2\pi]$ such that 
\begin{equation}\label{LipschitzGamma}
|\Gamma(v-z,\phi) -\Gamma(v'-z', \phi +\phi_0(z,v,z',v'), \phi)|\leq 3 |z-v-(z'-v')|
\end{equation}
and hence
\begin{equation}\label{Lipschitzalpha}
|\alpha(z,v,\theta,\phi) -\alpha(z', v', \theta, \phi +\phi_0(z,v,z',v', \phi))|\leq 2 \theta (|z-z'| +|v-v'|)
\end{equation}
and 
\begin{equation}\label{growthzalpha}
|\alpha(z,v,\theta,\phi)| \leq 2 \theta (|z| +|v|).
\end{equation}
\end{lemma}
   
From \eqref{Lipschitzalpha} it follows

     \begin{eqnarray}  & \int_0^\pi \left|  \int_0^{2\pi} {\alpha}(z,v, \theta,\phi) - {\alpha}(z',v', \theta, \phi)d\phi\right|Q(d\theta) 
   \nonumber \\&\leq C (|z-z'|+|v-v'|).\label{B3} 
   \end{eqnarray}
where by an abuse of notation we used the same constant $C>0$ for \eqref{growth} and \eqref{B3}.

\smallskip \noindent 
To write  \eqref{eqn1.1}, \eqref{eqn1.2} in its weak  form (in the sense of partial differential equations), we apply  a result of 
Tanaka \cite {Ta}. Let us denote by $ C(\mathbb{R}^{d})$, resp. $C_0(\mathbb{R}^{d})$, resp. $C_0^n (\mathbb{R}^{d})$ the space of real valued continuous functions, resp.continuous functions with compact support, resp. $n$ differentiable functions with compact support  on $\mathbb{R}^{d}$.
\begin {proposition} \label {PropAppTanaka}  
Let $\psi(u,v,y,z ) \in C_0(\mathbb{R}^{12})$, as a function of  $u,v,y,z \in \mathbb{R}^3$.  For each $\theta \in (0, \pi]$ fixed 
 
\begin {equation}
\int_{\mathbb{R}^{6}\times [0,2 \pi)} \psi(u,v,y^\star,z^\star)  d\phi du dv
= \int_{\mathbb{R}^{6}\times [0,2 \pi)} \psi(u^\star,v^\star,y,z)  d\phi du dv \label{eqn1.9}
\end {equation}

\end {proposition}

\begin{Remark}\label{RemExtensionPropTanaka} The  result in Proposition \ref{PropAppTanaka}    is proven in  the Appendix of   \cite{Ta}, by using equation \eqref{eqn1.4}. Since $ C_0(\mathbb{R}^{12})$ is dense in  the space of integrable functions $L^1 (\mathbb{R}^{12})$ the result can be extended to continuous  and integrable  functions $\psi(u,v,y,z ) \in C(\mathbb{R}^{12})\cap L^1 (\mathbb{R}^{12}) $.
\end{Remark}

By integration of the Boltzmann equation (\ref{eqn1.1}) with a test function $\Psi \in C^2_0(\mathbb{R}^6)$, the derivation of the weak formulation of the Boltzmann equation is obtained  in a similar way as in Tanaka \cite {Ta}, \cite{Ta3}, who treated the spatially homogeneous case, with $\sigma \equiv 1$. 

\smallskip \noindent 
In fact, consider the Boltzmann equation \eqref{eqn1.1} with collision operator \eqref{eqn1.2}. We multiply \eqref{eqn1.1} by a function  $\psi$  (of $(x, z) \in \mathbb{R}^6$)
 belonging to  $C^2_0(\mathbb{R}^6)$, 
and integrate with respect to $x$ and $z$, using integration by parts and \eqref{eqn1.9},   we arrive at the weak formulation  of the Boltzmann equation: 
\begin{align}
&\int_{\mathbb{R}^6} \psi(x,z) \frac{ \partial f}{\partial t} (t, x, z)dxdz - \int_{\mathbb{R}^6} f(t, x, z)(z,  \nabla_x \psi(x,z)) dxdz \notag\\
&= \int_{\mathbb{R}^6} f(t, x, z) L_{f(t)}\psi(x,z) dxdz \label {WBE}
\end {align}
 for all $t\in \mathbb{R}_+$  with
\begin{equation}\label{generatornonlinear}
 L_{f(t)}\psi(x,z)= \int_{\mathbb{R}^3\times (0,\pi]\times [0,2 \pi)} \{\psi(x,z^\star)- \psi(x,z)\}f(t, x, v)B(z, dv, d\theta)d\phi,
\end{equation}
where $B$ is as in \eqref{eqn1.8}.

\noindent The integral form  of equation \eqref{WBE} corresponds to 
\begin{align}
&\int_{\mathbb{R}^6} \psi(x,z)  f(t, x, z)dxdz = \int_{\mathbb{R}^6} \psi(x,z)  f(0, x, z)dxdz\notag\\&+ \int_0^t\int_{\mathbb{R}^6} f(s, x, z)(z,  \nabla_x \psi(x,z)) dxdzds 
+\int_0^t\int_{\mathbb{R}^6} f(s, x, z)  L_{f(s)}\psi(x,z) dxdz ds, \label {IWBE}
\end {align}

  \begin{dfn}\label{solBE} A collection  of densities $\{f(t,x,z)\}_{t \in \mathbb{R}_+}$,  with $x,z$ $\in \mathbb{R}^3$, is a strong (resp. weak)  solution of the Boltzmann equation in $[0,T]$ if for any $t\in [0,T]$ it solves  \eqref{eqn1.1} (resp.  \eqref{IWBE} for all $\psi \in$ $C^2_0(\mathbb{R}^6)$). 
  \end{dfn}
   We remark that a  strong solution of the Boltzmann equation is obviously also a weak solution of the Boltzmann equation.

\par
 We denote with  $\mathbb{D}:= \mathbb{D}(\mathbb{R}_+,\mathbb{R}^3)$ the  space of all  right continuous functions with left limits defined  on 
$[0, \infty)$ taking values in $\mathbb{R}^3$, and equipped with the topology induced by the Skorohod metric (see e.g. \cite {Bi}). 

   \begin{dfn}\label{Denskog}
        Let $T > 0$. Let   $(X_s,Z_s)_{s \in [0, T]}$  be a stochastic process with paths on  $\mathbb{D}\times \mathbb{D}$, and with time-marginal  
  densities $\{f(t,x,z)\}_{t \in [0, T]}$. Then $(X_s,Z_s)_{s \in [0, T]}$ is called a "Boltzmann process" if the collection of densities $\{f(t,x,z)\}_{t \in [0, T]}$ solves equation \eqref{IWBE}. 
           \end{dfn}    
We remark that the Fokker-Planck equation of a Boltzmann process is the Boltzmann equation, and the infinitesimal generator of a Boltzmann process is given by $ (z, \nabla_x) + L_{f(t)}$. 

\par
Costantini and Marra \cite{ConstantiniMarra1992} analyzed
hydrodynamical limits of a process given by a the drift term involving both $ (z, \nabla_x)$ and  $ L_{f(t)}$ in addition to a martingale term. Such a process in 
\cite{ConstantiniMarra1992} is not a Boltzmann process according to Definition \ref{Denskog}. 

\par
Stochastic analytic methods for studying the spatially homogeneous Boltzmann equation, which is obtained by integrating the Boltzmann equation over the space variable $x$, were initiated for Maxwellian molecules (non-cutoff case) by Tanaka \cite{Ta} \cite{Ta2} \cite{Ta3} even though Gr\"unbaum had earlier established the propagation of chaos result when $Q(0, \pi]$ is finite, i.e., for the cutoff case. The spatially homogeneous Boltzmann equation for other types of potentials and related problems were studied by several authors such as Funaki \cite{F1} \cite{F2}, Horowitz and Karandikar \cite{HK90}, Fournier \cite{F}, Fournier and Mouhot \cite{FM}, Fournier and M\'el\'eard \cite{FM01}, Lu and Mouhot \cite{LM12}, and Villani \cite{V}. 

\par
However, the equation as given by Boltzmann \cite{Bo} (also see \cite{Ce4} \cite{CIP}) is non-homogeneous i.e., depends on position and velocity, and poses a challenging problem to cast and study it by means of stochastic analysis. A rigorous derivation of the Boltzmann equation from a microscopic model, and the existence, uniqueness and properties of  solutions to  the Boltzmann equation presents many difficult and open problems. To overcome some of the difficulties, Morgenstern \cite{Mo} ``mollified" $Q(f, f)$ by replacing it by 
\begin{align*}
&Q_M(f, f)(t,x,u) \\
& = \int \{f(t, x, u^\star) f(t, y, v^\star)-f(t, x, u) f(t, y, v)\}K_M(x, y)B(u, dv, d\th) dy d\phi
\end{align*}
with some measurable $K_M$ and $B$    such that $K_M(x, y)B(u, dv, d\th)$ has a bounded density with respect to Lebesgue measure $dv \times d\th$, and 
obtaining a global existence theorem in $L^1(\R^3 \times \R^3)$. Povzner \cite{Pov} obtained existence and uniqueness, under suitable conditions, in the space of Borel measures in $x$, with  the term $K_M(x, y)B(u, dv, d\th) dy d\phi$ replaced by $K_P(x-y, u-v)dv dy$ (with  a suitable reinterpretation of the relations between 
$x, u^*, v^*$ and $y, u, v$, and suitable moments assumptions on $K_P$). 

\par
According to \cite{Ce4} (p. 399), this modification of Boltzmann's equation by Povzner is ``{\it close to physical reality}".  (cf. Rezakhanlou \cite{Re}). In a joint work with Albeverio \cite{ARS}, we considered the Boltzmann-Enskog equation, i.e., equation \eqref{eqn1.1}  with
\begin{align}
&Q_E^{\beta}(f, f) (t,x,u) \notag\\
& = \int_{\Lambda} \int_{\R^3} \{f(t, y, u^*)f(t, x, v^*) - f(t, y, u)f(t, x, v)\} \beta(|x - y|)dy B(u, dv, d\th)d\phi \label{collker}
\end{align}
where $\beta \in C_0(\mathbb{R}^1)$. 
 Moreover, we identified the Boltzmann-Enskog process as a solution of a McKean-Vlasov SDE, i.e.,  identified the stochastic process for which the associated Fokker-Planck equation is the Boltzmann-Enskog equation which is given by equation \eqref{eqn1.1} with collision kernel \eqref{collker}.
 Boltzmann-Enskog processes were constructed as the limit of suitable interacting particle systems \cite{FRS1}, and their  uniqueness and continuity properties were studied by us \cite{FRS2} jointly with Friesen. 
 
 \par
 The existence of solutions to the Boltzmann equation on a finite time interval has been established by Di Perna and Lions \cite{DL}.
 In the present article, we identify Boltzmann processes for which the associated  Fokker-Planck equation is the Boltzmann equation \eqref{eqn1.1} for hard spheres. This is accomplished by first assuming that the Boltzmann equation has a probability density-valued solution 
 $\{f(t, x, z)\}_{t \in [0, T]}$.   Using it, we prove the existence of a stochastic process such that its law (or probability distribution) $\{\mu_t\}_{t \in [0, T]}$ is a solution of a bilinear Boltzmann equation \eqref{weakbilinearBEmeasures} that involves $f$ and is introduced in this article. When $\mu_t$ admits a density $g(t,x, z)$ for each $t$, we show that $g(t, x, z) = f(t, x, z)$ a.e. under suitable conditions on $f$ and $g$. (In  \cite{ARS2} we gave a slightly different proof of this assertion.) Since $g$ coincides with $f$, the bilinear Boltzmann equation associated with $f$ 
 \eqref{weakbilinearBEmeasures} has a unique solution $\{f(t, x, z)\}_{t \in [0, T]}$, which also solves the Boltzmann equation \eqref{eqn1.1} by our assumption. 
 
 \par 
 The organization of the paper is as described below: In Section \ref{Sec bilinear Boltzmann eq}, we introduce a ``bilinear Boltzmann equation'' \eqref{eqn1.1lin} associated to a solution of  $\{f(t,x,z)\}_{t \in [0,T]}$ of the Boltzmann equation. In   Theorem \ref{thm-uniqueness IWlinBE-given-f} we show that  its only solution is $\{f(t,x,z)\}_{t \in [0,T]}$, i.e. the unique solution of the bilinear Boltzmann equation associated to $\{f(t,x,z)\}_{t \in [0,T]}$ is $\{f(t,x,z)\}_{t \in [0,T]}$ itself. This suggests that, once  a solution $\{f(t,x,z)\}_{t \in [0,T]}$ of the Boltzmann equation is given, the associated Boltzmann process can be found by first solving a  Poisson type  SDE whose Fokker -Planck equation is the bilinear Boltzmann equation associated to $\{f(t,x,z)\}_{t \in [0,T]}$.\\
 
In Subsection \ref{Par Invariants BE}, we recall that the kinetic energy of the Boltzmann model is conserved and  as a consequence,  in Lemma  \ref{Lemmaconserved}, prove an inequality for the bilinear Boltzmann equation   \eqref{eqn1.1lin}. This will allow us in Theorem \ref{Th moment control} in  Section \ref{Sec Existence Boltzmann process} to state in particular that the second moment of the SDE associated to the bilinear Boltzmann equation is finite.\\
           In Section \ref{Sec Boltzmann process}, Theorem \ref{Ito}, we construct a McKean-Vlasov stochastic differential equation (McKean-Vlasov SDE) solved by   the Boltzmann process, i.e. 
 we construct the McKean -Vlasov SDE  (\eqref{eq-spaceB},\eqref{eq-velB})  associated 
to equation \eqref{WBE} in $[0,T]$  and  prove  that any solution  $(X_t,Z_t)_{t\in [0,T]}$ is a Boltzmann process.\\
In Section \ref{Sec Existence Boltzmann process} we find sufficient conditions for the existence of  the solution of  (\eqref{eq-spaceB},\eqref{eq-velB}). The main theorem in this section  is Theorem \ref{mainThexistenceBoltzmann}, wherein given a strong  solution $\{f(t,x,z)\}_{t \in [0,T]}$ 
of the Boltzmann equation  \eqref{eqn1.1}, we find sufficient conditions for the  existence of a solution of the   McKean -Vlasov equation (\eqref{eq-spaceB}, \eqref{eq-velB}) with density $\{f(t,x,z)\}_{t \in [0,T]}$. The solution process $\{X_t,Z_t\}_{t\in [0,T]}$ is then a  Boltzmann process. \\ The proof of Theorem \ref{mainThexistenceBoltzmann}  is carried out by first analyzing the existence of the SDE associated to the bilinear Boltzmann equation \eqref{eqn1.1lin} in Subsection \ref{Par Construction linear BP}. It depends on other results stated and proven in Section \ref{Sec Existence Boltzmann process}. \\

\section{The bilinear Boltzmann equation}\label{Sec bilinear Boltzmann eq}
In this section we introduce a class of ''bilinear Boltzmann equations'' and discuss how their solutions are  related to the solutions of the Boltzmann equation. 

 We introduce the following definitions.

\begin{dfn}\label{DeflinBE} Let $T>0$.
Let  $\{f(t,x,z\}_{t \in \mathbb{R}_+}$  be  a strong  solution of the Boltzmann equation in $[0,T]$. 
 A collection  of densities $\{g(t,x,z\}_{t \in \mathbb{R}_+}$,  with $x,z$ $\in \mathbb{R}^3$, is a   strong  solution of the ''bilinear Boltzmann equation related to $\{f(t,x,z\}_{t \in \mathbb{R}_+}$''  in $[0,T]$ , if for any $t\in [0,T]$  it solves 
\begin{equation}
\frac{ \partial g}{\partial t} (t, x, z)+ z \cdot \nabla_x g(t, x, z) = Q(f, g)(t,x,z), \label{eqn1.1lin}
\end{equation}
with 
\begin{equation}
Q(f, g) (t,x,z)= \int_{\R^3\times \Xi} \{f(t, x, z^\star) g(t, x, v^\star)-f(t, x, z) g(t, x, v)\}B(z, dv, d\theta)d\phi. \label{eqn1.2lin}
\end{equation}

\end{dfn}

\begin{dfn}\label{DefIWlinBE} Let $T>0$.
Let  $\{f(t,x,z\}_{t \in \mathbb{R}_+}$  be  a strong  solution of the Boltzmann equation in $[0,T]$. 
 A collection  of densities $\{g(t,x,z\}_{t \in \mathbb{R}_+}$,  with $x,z$ $\in \mathbb{R}^3$, is a   weak  solution of the ''bilinear Boltzmann equation related to $\{f(t,x,z\}_{t \in \mathbb{R}_+}$''  in $[0,T]$ , if for any $t\in [0,T]$ and  all $\psi \in$ $C^2_0(\mathbb{R}^6)$ it solves 
 \begin{align}
&\int_{\mathbb{R}^6} \psi(x,z)  g(t, x, z)dxdz = \int_{\mathbb{R}^6} \psi(x,z)  g(0, x, z)dxdz\notag\\&+ \int_0^t\int_{\mathbb{R}^6} g(s, x, z)(z,  \nabla_x \psi(x,z)) dxdzds + \int_0^t  Q_s(f,g)(\psi) ds 
\end {align}
with $ Q_s(f,g)$ being the operator 
  \begin{equation} Q_s(f,g)(\psi):=
\int_{\mathbb{R}^6} g(s,x,z)   L_{f(s)}\psi(x,z)  dxdz. \label{OperatorQbi} 
\end{equation} 
\end{dfn}
Obviously any   strong  solution of the ''bilinear Boltzmann equation related to $\{f(t,x,z\}_{t \in \mathbb{R}_+}$''  in $[0,T]$ is also a   weak  solution of the ''bilinear Boltzmann equation related to $\{f(t,x,z)\}_{t \in \mathbb{R}_+}$''  in $[0,T]$, but not vice versa.\\

 Let   $\{\mu_t(dx,dz)\}_{t \in {\R}_+^0}$ be a collection of probability measures  on $(\R^3\times \R^3, \mathcal{B}(\R^3\times \R^3))$. Let us define the operator 
 \begin{equation} Q_s(f,\mu)(\psi):=
\int_{\mathbb{R}^6}   L_{f(s)}\psi(x,z) \mu_s(dx, dz). \label{OperatorQProbmeasure}\end{equation} 
acting on all $\psi$ for which the integral on the right side is finite.  Let us remark that (by abuse of notation)    $Q_s(f,\mu)=  Q_s(f,g)$ if $\mu_s=g(s,x,z)dx dz$.
\begin{dfn}\label{DefIWlinBEmeasures}
Let  $\{f(t,x,z)\}_{t \in \mathbb{R}_+}$  be  a strong  solution of the Boltzmann equation in $[0,T]$. 
 A collection  $\{\mu_t(dx,dz)\}_{t \in {\R}_+^0}$  of probability measures on $(\R^3\times \R^3, \mathcal{B}(\R^3\times \R^3))$, is a   weak  solution of the ''bilinear Boltzmann equation related to $\{f(t,x,z)\}_{t \in \mathbb{R}_+}$''  in $[0,T]$ , if for any $t\in [0,T]$ and  all $\psi \in$ $C^2_0(\mathbb{R}^6)$ it solves 
 \begin{align}\label{weakbilinearBEmeasures}
&\int_{\mathbb{R}^6} \psi(x,z)  \mu_t(dx,dz) = \int_{\mathbb{R}^6} \psi(x,z)  \mu_0(dx,dz)\notag\\&+ \int_0^t\int_{\mathbb{R}^6} (z,  \nabla_x \psi(x,z)) \mu_s(dx,dz)ds + \int_0^t  Q_s(f,\mu)(\psi) ds 
\end {align} 
\end{dfn}

\subsection{The relation between bilinear Boltzmann equation and Boltzmann equation}\label{Par bilinear BE}

In order to establish the relationship between the equations \eqref{eqn1.1lin} and \eqref{eqn1.1}, we introduce the following hypothesis.

 {\bf C1} The densities $f(t, x, z)$ and $g(t, x, z)$ are in $C^{1,2}([0, T] \times \mathbb{R}^6)$ and are strictly positive-valued a.e. with 
$ g\log g, g\log f \in L^1(\mathbb{R}^6)$ for each $t \in [0, T]$ and  $\lim_{|x|\to \infty} g(t,x,z)=0$

 \begin{theorem} \label{thm-uniqueness IWlinBE-given-f}   
Assume  the collection of densities  $\{f(t,x,z\}_{t \in \mathbb{R}_+}$ is  a strong  solution of the Boltzmann equation in $[0,T]$ and the  collection  of densities $\{g(t,x,z\}_{t \in \mathbb{R}_+}$ is  a  strong  solution of the ''bilinear Boltzmann equation related to $\{f(t,x,z\}_{t \in \mathbb{R}_+}$''  in $[0,T]$. Let the hypothesis {\bf C1} be satisfied  and $g(0,x,z) = f(0,x,z)$ a.e.\\
Then $g(t,x,z)=$ $f(t,x,z)$ $\quad a.e.$ for all $t\in [0,T]$.  
 \end{theorem} 

\begin{proof} 
To prove Theorem \ref{thm-existence-given-f}, it suffices to show the following equality 
  written in terms of relative entropy 
 \begin{equation} \label{relative entropy eq}
 R_t(g|f):= \int_{\mathbb{R}^6} \ln{\left(\frac{g(t.x.z)}{f(t,x,z)}\right)} g(t,x,z) dxdz = 0 \quad \forall \,t\in [0,T].
 \end{equation}
Indeed, since $\{g(t,x,z\}_{t \in \mathbb{R}_+}$ is a collection of densities, it follows that
  \begin{align*}
  &\int_0^t  \int_{\mathbb{R}^6 } g(s,x,z) \frac{\partial}{ \partial s}  \ln{(g(s,x,z))} dx dz ds =  \int_0^t  \int_{\mathbb{R}^6 }\frac{\partial}{ \partial s} g(s,x,z) dx dz ds  \notag \\&=  \int_{\mathbb{R}^6 }(g(t,x,z) -g(0,x,z)) dxdz =0\notag 
  \end{align*}
 Integration by parts yields 
  \begin{align*}
&\int_0^t  \int_{\mathbb{R}^6 } \ln{(g(s,x,z))}  \frac{\partial}{ \partial s}   g(s,x,z)dx dz ds =\\& \int_{\mathbb{R}^6 }(\ln{(g(t,x,z))}g(t,x,z) - \ln{(g(0,x,z))}g(0,x,z)) dx dz
\end{align*}
Next, consider 
\begin{align}\label{ggradln}
  &  \int_0^t\int_{\mathbb{R}^6 }  g(s,z,x)  z \cdot \nabla_x \ln{(g(s,x,z))} dxdz\notag \\& =\int_0^t\int_{\mathbb{R}^6 }  z \cdot \nabla_x  g(s,z,x) dxdz =0 
  \end{align}
  where the last equality is obtained from integrating by parts, and 
   the condition  $\lim_{|x|\to \infty} g(t,x,z)=0$. 
  From \eqref{ggradln} we get 
\begin{equation}\label{gradgln}
  \int_0^t\int_{\mathbb{R}^6 }  \ln{(g(s,x,z))}  z \cdot \nabla_x  g(s,z,x)  dxdzds = 0
  \end{equation}
since $\int_0^t \int_{\mathbb{R}^6} \nabla_x( z \cdot g(s, x, z)\ln g(s, z, x))dxdzds
= 0$ by {\bf{C1}}.
Using the fact that  $\{g(t,x,z\}_{t \in \mathbb{R}_+}$ is  a  strong  solution of the bilinear Boltzmann equation related to $\{f(t,x,z\}_{t \in \mathbb{R}_+}$  in $[0,T]$, it follows 
 \begin{align} \label{lng}
  & \int_{\mathbb{R}^6} \ln{(g(t,x,z))} g(t,x,z) dxdz - \int_{\mathbb{R}^6} \ln{(g(0,x,z))}  g(0,x,z) dxdz \notag \\& 
 =\int_0^t \int_{\mathbb{R}^6} \ln{(g(t,x,z))}  \frac{\partial}{\partial s} g(s,x,z) dxdzds\notag\\&
=\int_0^t\int_{\mathbb{R}^6} \ln{(g(s,x,z))} Q(f,g)(s,x,z) dx dz ds\notag\\&
 =  \int_0^t  \int_{\mathbb{R}^9 \times \Xi} \{  \ln{(g(s,x,z^\star))}- \ln{(g(s,x,z))}\}\notag \\&\times \sigma(|z-v|) f(s,x,v)g(s,x,z)Q(d\theta)d\phi dv  dxdzds 
 \end{align}
where we have used Proposition \ref{PropAppTanaka}, Remark \ref{RemExtensionPropTanaka} and equation \eqref{gradgln}.

Along similar lines and condition {\bf C1}, we also get 
 \begin{align}\label{IWlinBEforlngf}
& \int_0^t \int_{\mathbb{R}^6} \ln{(f(t,x,z))}  \frac{\partial}{\partial s} g(s,x,z) dxdzds\notag\\&
= \int_0^t \int_{\mathbb{R}^6} \ln{(f(s,x,z))} Q(f,g)(s,x,z) dx dz ds\notag\\&
+  \int_0^t \int_{\mathbb{R}^6} g(s,x,z)    z \cdot \nabla_x  \ln{(f(s,x,z))} dx dz ds\notag\\&
 =  \int_0^t  \int_{\mathbb{R}^9 \times \Xi} \{  \ln{(f(s,x,z^\star))}- \ln{(f(s,x,z))}\}\notag \\&\times \sigma(|z-v|) f(s,x,v)g(s,x,z)Q(d\theta)d\phi dv  dxdzds \notag\\&
+  \int_0^t \int_{\mathbb{R}^6} g(s,x,z)    z \cdot \nabla_x  \ln{(f(s,x,z))} dx dz ds.
 \end{align}
Using 
\begin{align*}
  & \int_{\mathbb{R}^6} \ln{(f(t,x,z))} g(t,x,z) dxdz -
  \int_{\mathbb{R}^6} \ln{(f(0,x,z))}  g(0,x,z) dxdz \notag \\&
=\int_0^t  \int_{\mathbb{R}^6} (g(s,x,z)\frac{\partial}{\partial s}\ln{(f(s,x,z))}  + \ln{(f(s,x,z))} \frac{\partial}{\partial s}g(s,x,z) )dxdzds
  \end{align*}
and recalling that  $\{f(t,x,z)\}_{t\in [0,T]}$ is  a  collection of densities which solves the Boltzmann equation 
\eqref{eqn1.1}, we obtain 
\begin{align}
  & \int_{\mathbb{R}^6} \ln{(f(t,x,z))} g(t,x,z) dxdz -
  \int_{\mathbb{R}^6} \ln{(f(0,x,z))}  g(0,x,z) dxdz \notag \\&
 = \int_0^t  \int_{\mathbb{R}^6 } \frac{Q(f, f) (s,x,z)} {f(s,x,z)} g(s,x,z) dxdzds \notag \\& + \int_0^t  \int_{\mathbb{R}^6 \times \mathbb{R}^3 \times \Xi} \{  \ln{(f(s,x,z^\star))}- \ln{(f(s,x,z))}\} \label{InfIto}\\&\times \sigma(|z-v|) f(s,x,v)g(s,x,z)Q(d\theta)d\phi dv  dxdzds \notag
 \end{align}
 The explicit form of the first term on the right side in \eqref{InfIto} is given by
 \begin{align}
  & 
\int_0^t  \int_{\mathbb{R}^9  \times \Xi} \{  f(s,x,z^\star) f(s,x,v^\star)- f(s,x,z)f(s,x,v)\}\notag \\&\times \sigma(|z-v|)\frac{g(s,x,z)} { f(s,x,z)}Q(d\theta)d\phi dv  dxdzds \notag
\\& = 
\int_0^t  \int_{\mathbb{R}^9  \times \Xi} \{  \frac{g(s,x,z^\star)} {f(s,x,z^\star)}- \frac{g(s,x,z)}{f(s,x,z)}\}\notag \\&\times \sigma(|z-v|)f(s,x,v)  f(s,x,z)Q(d\theta)d\phi dv  dxdzds\,, \notag
\end{align}  
where in the last equation we have used Proposition \ref{PropAppTanaka}. Equation \eqref{InfIto}  then becomes
\begin{align}\label{lnf}
  & \int_{\mathbb{R}^6} \ln{(f(t,x,z))} g(t,x,z) dxdz -
  \int_{\mathbb{R}^6} \ln{(f(0,x,z))}  g(0,x,z) dxdz \notag \\&
 = \int_0^t  \int_{\mathbb{R}^9  \times \Xi} \{  \frac{g(s,x,z^\star)} {f(s,x,z^\star)}- \frac{g(s,x,z)}{f(s,x,z)}\}\notag \\&\times \sigma(|z-v|)f(s,x,v)  f(s,x,z)Q(d\theta)d\phi dv  dxdzds\,, \notag \\&+ \int_0^t  \int_{\mathbb{R}^6 \times \mathbb{R}^3 \times \Xi} \{  \ln{(f(s,x,z^\star))}- \ln{(f(s,x,z))}\}\notag \\&\times \sigma(|z-v|) f(s,x,v)g(s,x,z)Q(d\theta)d\phi dv  dxdzds 
\end{align}  
 Combining equations \eqref{lng} with \eqref{lnf} we obtain the following equality for relative entropy
  \begin{align} 
& R_t(g|f)= \notag\\& \int_{\mathbb{R}^6} \ln{(g(t,x,z))} g(t,x,z) dxdz -
  \int_{\mathbb{R}^6} \ln{(f(t,x,z))}  g(t,x,z) dxdz \notag \\&
 = \int_0^t  \int_{\mathbb{R}^9  \times \Xi} \{  \ln{\left(\frac{g(s,x,z^\star)} {f(s,x,z^\star)}\right)}- \ln{\left(\frac{g(s,x,z)}{f(s,x,z)}\right)}\}\notag \\&\times \sigma(|z-v|)f(s,x,v)  g(s,x,z)Q(d\theta)d\phi dv  dxdzds\,, \notag\\& - \int_0^t  \int_{\mathbb{R}^9  \times \Xi} \{  \frac{g(s,x,z^\star)} {f(s,x,z^\star)}- \frac{g(s,x,z)}{f(s,x,z)}\}\notag \\&\times \sigma(|z-v|)f(s,x,v)  f(s,x,z)Q(d\theta)d\phi dv  dxdzds\,, \notag
\end{align}  
or the equivalent equation 
 \begin{align} \label{relentropy2}
& R_t(g|f)= \notag\\& 
  \int_0^t  \int_{\mathbb{R}^9  \times \Xi} \left[\frac{g(s,x,z)}{f(s,x,z)} \{ 1+ \ln{\left( \frac{g(s,x,z^\star)/f(s,x,z^\star)} {g(s,x,z)/f(s,x,z)} \right) }\}-\frac{g(s,x,z^\star)} {f(s,x,z^\star)} \right]
  \notag \\&\times \sigma(|z-v|)f(s,x,v)  f(s,x,z)Q(d\theta)d\phi dv  dxdzds \,.
  \end{align}
  Transforming  \eqref{relentropy2} in the   equivalent equation below, and recalling that for $x \ge 0$ we have $1+\ln{(x)} -x \leq 0$,   we  easily see   that $ R_t(g|f) \leq 0$:
   \begin{align} \label{relentropyineq}
& R_t(g|f)= \notag\\& 
  \int_0^t  \int_{\mathbb{R}^9  \times \Xi} \left[  1+ \ln{\left( \frac{g(s,x,z^\star)/f(s,x,z^\star)} {g(s,x,z)/f(s,x,z)} \right) }-\frac{g(s,x,z^\star)/f(s,x,z^\star)}{g(s,x,z)/f(s,x,z)} \right]
  \notag \\&\times \frac{g(s,x,z)}{f(s,x,z)} \sigma(|z-v|)f(s,x,v)  f(s,x,z)Q(d\theta)d\phi dv  dxdzds  \notag \\& \leq 0\,.
  \end{align}
  Since $R_t(g|f) \geq 0$, it follows the identity $R_t(g|f) = 0$

\end{proof}

\subsection {Invariants of the  Boltzmann equation and consequences for the bilinear form}     \label{Par Invariants BE}

Let  $\{f(t,x,z\}_{t \in \mathbb{R}_+}$  be  a strong  solution of the Boltzmann equation in $[0,T]$. And for $t\in [0,T]$ consider 

\begin{equation} Q_t(f,f)(\psi):=
\int_{\mathbb{R}^6} f(t, x, z)  L_{f(t)}\psi(x,z) dxdz \label{OperatorQ}\end{equation} 
 defined through the right side  of equation \eqref{WBE}. \\
Using \eqref{eqn1.3}, as well as Proposition \ref{PropAppTanaka} it is easy to verify that 
\begin{equation} Q_t(f,f)(\psi)=0\quad \text{for}\quad \psi(x,z)=a+(b,z)+c |z|^2 
\label{invfunction}\end{equation}
$\, \forall \, a,c\in \mathbb{R} \,, b \in \mathbb{R}^3$. (For a proof see  Chapter II.7 \cite{Ce} or \cite{CIP}, \cite{Bressan})
As a consequence of \eqref{IWBE} and \eqref{invfunction} we obtain that the  first and second moment of a Boltzmann process is conserved, i.e. 
\begin{equation} 
\left\{
\begin{aligned}
\int_{\mathbb{R}^6} z  f (t, x, z)dxdz= \int_{\mathbb{R}^6} z   f (0, x, z)dxdz,  \\
\int_{\mathbb{R}^6} |z|^2   f (t, x, z)dxdz= \int_{\mathbb{R}^6} |z|^2   f (0, x, z)dxdz.
\end{aligned}
\right. \label{conservation moments}
\end{equation}

 Equation \eqref{invfunction} was derived by L. Boltzmann in much more generality. By 
 taking into account a second   collection of densities  $\{g(t,x,z)\}_{t \in {\R}_+^0}$ on $(\R^3\times \R^3, \mathcal{B}(\R^3\times \R^3))$, Boltzmann obtained 
 \begin{equation}  Q_t^S(f,g)(\psi)=0\quad \text{for}\quad  \psi(x,z)=a+(b,z)+c |z|^2 
\label{generalinvfunction}\end{equation}
$\forall \, a,c\in \mathbb{R} \,, b \in \mathbb{R}^3$
with the operator $Q_t^S$ defined through
\begin{equation} 
 Q_t^S(f,g)(\cdot):=Q_t(f,g)(\cdot)+Q_t(g,f)(\cdot) \label{SymmQ}\end{equation}
 with $ Q_t(f,g)(\psi)$ defined through \eqref{OperatorQbi}.
 For the history of the derivation of \eqref{invfunction} or \eqref{SymmQ} and their proofs, we refer to \cite{CIP}, Paragraph {3.1}.\\

We add here the following result:

 \begin{lemma} \label{Lemmaconserved}
\begin{align}\label{consKineticEnergyLin}
&Q_t(f,\mu)(|z|^2)=\notag\\&-\int_{\mathbb{R}^9\times \Xi} (|z|^2-|v|^2)\sigma(|z-v|) \sin^2(\frac{\theta}{2})Q(d\theta)d\phi  f(t,x,v)  dv \mu_t(dx, dz)
\end{align}
\end{lemma}
with $Q_t(f,\mu)$ defined in  \eqref{OperatorQProbmeasure}.

\begin{proof}
\begin{align} & Q_t(f,\mu)(|z|^2)= \int_{\mathbb{R}^6}   L_{f(t)} |z|^2 \mu_t(dx, dz) \label{linearizedRHSenergy}\\& =
 \int_{\mathbb{R}^9\times \Xi} (|z^\star|^2-|z|^2) \sigma(|z-v|)Q(d\theta) d\phi f(t,x,v) dv \mu_t(dx, dz). \notag
\end{align}  

Moreover, by adding and subtracting the same quantity,  we write
\begin{align} & Q_t(f,\mu)(|z|^2)=\notag\\&
 \int_{\mathbb{R}^9\times \Xi} (|z^\star|^2+ |v^\star|^2-|z|^2 -|v|^2) \sigma(|z-v|)Q(d\theta) d\phi f(t,x,v)  dv \mu_t(dx, dz) \notag\\&
-  \int_{\mathbb{R}^9\times \Xi} (|v^\star|^2-|v|^2) \sigma(|z-v|)Q(d\theta) d\phi f(t,x,v) dv \mu_t(dx, dz)\notag
 \\= & -  \int_{\mathbb{R}^9\times \Xi} (|v^\star|^2-|v|^2) \sigma(|z-v|)Q(d\theta) d\phi f(t,x,v) dv \mu_t(dx, dz).  \label{linearizedRHSenergy2}
\end{align}  
 where in the last equality we used  that the kinetic energy is conserved during the elastic collision, see \eqref{eqn1.3}.\\
 Combining equation \eqref{linearizedRHSenergy} and \eqref{linearizedRHSenergy2}, we obtain 
 \begin{align}\label{Qzsquare} & Q_t(f,\mu)(|z|^2)=\notag\\& \frac{1}{2} \int_{\mathbb{R}^9\times \Xi} (|z^\star|^2-|z|^2 -(|v^\star|^2- |v|^2) \sigma(|z-v|)Q(d\theta) d\phi f(t,x,v)  dv \mu_t(dx, dz) \notag\\& =    
 \frac{1}{2} \int_{\mathbb{R}^9\times \Xi} (|z|^2+ 2(z,\alpha)+|\alpha|^2 -|z|^2) -(|v|^2- 2(v,\alpha)+|\alpha|^2-|v|^2)\notag\\& \times  \sigma(|z-v|)Q(d\theta) d\phi f(t,x,v)  dv \mu_t(dx, dz) \notag\\& =    
 \int_{\mathbb{R}^9\times \Xi} (z+v,\alpha)  \sigma(|z-v|)Q(d\theta) d\phi f(t,x,v) dv \mu_t(dx, dz), 
  \end{align}
  Using the parametrization \eqref{PrameterTrans} for $\alpha$$=\alpha(z,v,\theta,\phi)$  and \eqref{symmGamma}  we obtain  
 \begin{align*} & Q_t(f,\mu)(|z|^2)=\notag\\& 
   \int_{\mathbb{R}^9\times \Xi} (z+v,v-z) \sin^2(\frac{\theta}{2}) \sigma(|z-v|)Q(d\theta) d\phi f(t,x,v) dv \mu_t(dx, dz)\\&=-\int_{\mathbb{R}^9\times \Xi} (|z|^2-|v|^2)\sin^2(\frac{\theta}{2})\sigma(|z-v|) Q(d\theta)d\phi  f(t,x,v)  dv \mu_t(dx, dz)
  \end{align*}
  
 \end{proof} 
The ideas of the proof of Lemma \ref{Lemmaconserved} are incorporated in obtaining moment estimates  and showing that the second moment of a  solution $\{\mu_t(dx,dz)\}_{t \in {\R}_+^0}$ of \eqref{weakbilinearBEmeasures} remains  for $ t>0$ finite, if this holds at initial time $t=0$. We obtain this result in  inequality \eqref{secondmomentmeasure} in Section 4 which is proven by \eqref{expsupZjsquare}.
We will prove this  however   first   for an approximating solution  in Theorem \ref{Th moment control} in  subsection \ref{Par Construction linear BP}. In  Theorem \ref{Th moment control}, we also show that the bound on the energy does not depend on the approximation. This will be a fundamental step to prove in Theorem \ref{Theorem MR}  the existence of  the  solution  of  the SDE  (\eqref{eq-spacer}, \eqref{eq-velr})  associated to the bilinear Boltzmann equation, i.e. the existence of the stochastic process  with time marginals  $\{\mu_t(dx,dz)\}_{t \in {\R}_+^0}$ satisfying  \eqref{weakbilinearBEmeasures}.

\section { Boltzmann processes} \label{Sec Boltzmann process}
In this Section we construct   a McKean-Vlasov SDE associated to  
 the Boltzmann equation \eqref{IWBE} and characterize Boltzmann 
processes.\\

    We use the following notation: 
    \[
    U_0 =  \R^3 \times [0,\pi)\times (0,2\pi].\]

    Let  $\{f(t,x,v)\}_{t \in {\R}_+^0}$ be a collection of densities on $(\R^3\times \R^3, \mathcal{B}(\R^3\times \R^3))$. 
    Then  $m(t, v)$ denotes the marginal density of velocity $v$ at time $t$, i.e. $m(t,v):= \int_{\R^3} f(t,x,v) dx$  so that
    $f(t,x |v){m(t,v)}$$:={f(t,x,v)}$ upon disintegration of measures. \\
    We make the following 

\smallskip
\noindent
{\bf{Hypotheses B:}}
 We assume that  $t \,\to\, f(t,x,v)$ 
is differentiable for each $x,v$ $\in \R^3$ fixed, and satisfies 
\begin{itemize}
\item[\bf{B0.}]
\[|\frac {\partial f}{\partial t} |\,\, \text{is bounded on any  compact set  of}\,\, \R_+^0\times \R^6,\]
\item[\bf{B1.}]
\[\frac {\partial f}{\partial t}(t,\cdot) \in L^1(\R^6),\quad \forall t\in \R_+^0,\]
\item[\bf{B2.}] 
\[
\sup_{ x \in \R^3 } \int_{\R^3} |z|^{1+\gamma} f(s,x,z)dz \in C([0, T]) \quad \forall T>0. 
\]
\item[\bf{B3.}]
\[
\sup_{ x \in \R^3 } \int_{\R^3} |z|^{1+\gamma} \frac {\partial }{\partial t}  f(t,x,z) dz \in L^1([0, T]) \quad \forall T>0.
\]
\end{itemize}
    
\begin{theorem}\label{Ito}      Let  $\{f(t,x,z)\}_{t \in {\R}_+^0}$ be any collection of densities which satisfies  hypothesis {\bf {B0}} - {\bf {B3}}. 
Suppose hypothesis {\bf A} holds.  Let $X_0$ and $Z_0$ be $\mathbb{R}^3$- valued random variables. 
Suppose that for  fixed $T>0$ there exists a stochastic basis $(\Omega,\mathscr{F},(\mathscr{F}_t)_{t\in [0,T]},\mathrm{P})$,   
an adapted process $(X_t,Z_t)_{t\in [0,T]}$ with values on  $\mathbb{D}\times \mathbb{D}$, 
which has  time marginals with density $f(t,x,z)$,   and satisfies  a.s. the following stochastic equation for $t \in [0,T]$:
\begin{align}
&X_t = X_0+\int_0^t Z_s ds \label{eq-spaceB}   \\
 &Z_t = Z_0 + \int_0^t \int_{U_0\times \R^+_0}\alpha(Z_{s_-}, v, \theta, \phi)1_{[0,\; \s(|Z_{s_-} - v|)f(s,X_s|v)]}(r) 
 d {N}  \label{eq-velB}
\end{align}
    where in the above equation,  $d{N}:= {N}(dv,d\theta, d\phi,dr,ds)$ is a  Poisson random measure (cPrm)  
    with compensator $m(s,v) dv Q(d\theta)d\phi  ds dr $.  
    Then  $\{f(t,x,z)\}_{t \in {\R}_+^0}$ is a weak solution of the Boltzmann equation in $[0,T]$ and hence  $(X_t,Z_t)_{t\in [0,T]}$ is a Boltzmann process. 
   \end {theorem}
 \begin {proof} 
    From \eqref{growth} it follows  for each $T>0$ 
    \begin{align*}
     & \int_0^T  \mathbb{E}[ \int_{ U_0\times \R_+^0}  
    | \alpha(Z_s, v, \theta, \phi) | 1_{[0,\; \sigma(|Z_{s} - v|)f(s, X_s  | v) ]}(r) m(s,v) dv Q(d\theta)d\phi dr] ds \\& 
     =\int_0^T  \mathbb{E}[ \int_{ U_0}  
    | \hat{\alpha}(Z_s, v, \theta, \phi) |  f(s,X_s,v) dv Q(d\theta)d\phi ] ds\\& 
     \leq  C \int_0^T \int_{\R^9} (|z|^{1+\gamma}+|v|^{1+\gamma}) f(s,x,z) f(s,x,v) dx dz dv ds,
     \\&   \leq 2 C \int_0^T \sup_{x\in \R^3}  \int_{\R^6} \left(  |z|^{1+\gamma} f(s,x,z)  dz \right)f(s,x,v) dv dx  ds <\infty.
      \end{align*} 
      for some constant $C>0$.
      In the above estimates we used that  the function  $f$ is also the density of  the process $(Z, X)$, 
      as well the assumption  { \bf A2} and  { \bf B2}.
It follows that we can  apply the It\^o formula to $(X_s,Z_s)_{s \in \mathbb{R}_+}$ \cite{RZ}. 
 In fact let $t$, $\Delta t > 0$,  $\psi \in C_0^2(\R^3 \times \R^3)$,  then 
\begin{align*}
 & \psi(X_{t+ \Delta t},Z_{t+ \Delta t})\nonumber \\
 &= \psi(X_t,Z_t) +  \int_t^{t+ \Delta t} (Z_s,  \nabla_x \psi (X_s, Z_s)) ds \nonumber\\
 &+ \int_t^{t+\Delta t} \int_{ U_0\times \R^+_0} 
 \{\psi (X_s, Z_s + \alpha(Z_s,v,\theta,\phi)1_{[0,\; \s(|Z_{s} - v|)f(s,X_s|v)]}(r)) -\psi (X_s, Z_s)  \nonumber 
 \\ & -\nabla_z\psi (X_s,Z_s) \alpha(Z_s, v, \theta, \phi)1_{[0,\; \s(|Z_{s} - v|)f(s,X_s| v)]}(r)\} m(s,v) dv  Q(d\theta) d\phi ds dr\, +\nonumber\\
 &\int_t^{t+\Delta t} \int_{ U_0}   ( \alpha(Z_s, v, \theta, \phi), \nabla_z\psi (X_s,Z_s))\s(|Z_s - v|)f(s,X_s,v) 
  dv Q(d\theta) d\phi ds
 \nonumber\\
 & + M^{t+ \Delta t}_t(\psi) )\end{align*}
 \begin{align}\label{eqIto}
 &= \psi(X_t,Z_t) +  \int_t^{t+ \Delta t} (Z_s,  \nabla_x \psi (X_s, Z_s)) ds \nonumber\\
 &+ \int_t^{t+\Delta t} \int_{ U_0} \{\psi (X_s, Z_s + \alpha(Z_s,v,\theta,\phi)) -\psi (X_s, Z_s)  \nonumber 
 \\ & -\nabla_z\psi (X_s,Z_s) \alpha(Z_s, v, \theta, \phi)\} \s(|Z_{s} - v|)f(s,X_s| v) m(s,v) dv Q(d\theta) d\phi ds  \,+\nonumber\\
 &\int_t^{t+\Delta t} \int_{ U_0}   ( \alpha(Z_s, v, \theta, \phi), \nabla_z\psi (X_s,Z_s))\s(|Z_s - v|)f(s,X_s,v) 
  dv Q(d\theta) d\phi ds
 \nonumber\\
 & + M^{t+ \Delta t}_t(\psi)
 \nonumber \\
  &= \psi(X_t,Z_t) +  \int_t^{t+ \Delta t} (Z_s,  \nabla_x \psi (X_s, Z_s)) ds \nonumber\\
  &+ \int_t^{t+\Delta t} \int_{ U_0} \{\psi (X_s, Z_s + \alpha(Z_s,v,\theta, \phi)) -\psi (X_s, Z_s)  \} \s(|Z_{s} - v|)f(s,X_s,v) dv Q(d\theta) d\phi ds \\
 & + M^{t+ \Delta t}_t(\psi)
  \nonumber
 \end{align}
 where $\{M^{t+ \Delta t}_t(\psi)\}_{\Delta t \in (0,T]}\}$ is a martingale, for each $T\in \mathbb{R}$. It follows 
  \begin{align}
   & \mathbb{E}[\psi(X_{t+ \Delta t},Z_{t+ \Delta t})- \psi(X_t,Z_t)]=   
   \nonumber\\ &  \mathbb{E}\left[ \int_t^{t+ \Delta t} (Z_s,  \nabla_x \psi (X_s, Z_s)) ds \right] \nonumber\\
     &+  \mathbb{E} \left[\int_t^{t+\Delta t} \int_{ U_0} \{\psi (X_s, Z_s + \alpha(Z_s,v,\theta, \phi)) -\psi (X_s, Z_s)  \} \s(|Z_{s} - v|)f(s,X_s,v) dv Q(d\theta) d\phi ds \right]\nonumber
    \end{align}
 Upon dividing  by $\Delta t$ on both sides, we obtain 
  \begin{align}
  &\lim_{\Delta t \downarrow 0} \frac{1}{\Delta t}\int_{\R^6 }\psi(x,u) \{f(t+\Delta t,x,u) -f(t,x,u)\}dx du
  \nonumber \\&= \lim_{\Delta t \downarrow 0} \frac{1}{\Delta t} \int_t^{t+\Delta t} \int_{\R^6 } (u,  \nabla_x \psi (x, u)) f(s,x,u) dx du ds \,+
  \nonumber \\&
   \lim_{\Delta t \downarrow 0} \frac{1}{\Delta t} \int_t^{t+\Delta t} 
   \int_{{\R^6 \times \R^3 \times [0,\pi)\times (0,2\pi]}}\{\psi(x,u+\alpha(u,v,\theta,\phi))-\psi(x,u)\}
  \nonumber  \\&\times \s(|u-v|)f(s,x,v) f(s,x,u)  dv Q(d\theta) d\phi  dx du  ds \label{Itolimt}
  \end{align}
 
 Letting $\Delta t \to 0$ in every term of \eqref{Itolimt} we obtain \eqref{WBE} as shown below.\\
 Let us consider the first term on the right side of \eqref{Itolimt}. By hypothesis  {\bf B}, 
$f$  is bounded by a constant $M$  on the compact set $[t,t+ \epsilon]\times K$, where $K$ is the compact support of $\psi$, 
 and $\epsilon >0$ is fixed. It follows that on $[t,t+ \epsilon]\times K$ 
 \begin{equation*}
  |(u,\nabla_x \psi(x,u)|f(s,x,u)\leq M |u| |\nabla_x \psi(x,u)|.
 \end{equation*}
 Since $ |(u,\nabla_x \psi(x,u)|$$= |(u,\nabla_x \psi(x,u)|1_K(x,u)$ is an integrable  function over $\R^6$,  the Lebesgue dominated convergence theorem implies that 
 \begin{equation*}
  h(s):= \int_{\R^6 } (u,  \nabla_x \psi (x, u)) f(s,x,u) dx du 
 \end{equation*}
 is a continuous function in $[t,t+ \epsilon]$, because $s\to f(s,x,u)$ is continuous. It follows that 
  \begin{equation*}
   \lim_{\Delta t \downarrow 0} \frac{1}{\Delta t} \int_t^{t+\Delta t} h(s) ds=h(t).
  \end{equation*}
\\
 Let us consider the second term on the right side of \eqref{Itolimt} and prove the continuity of the function
 \begin{eqnarray*} & g(s):=
 \int_{{\R^6 \times \R^3 \times [0,\pi)\times (0,2\pi]}}\{\psi(x,u+\alpha(u,v,\theta,\phi))-\psi(x,u)\}
 \\& \times\s(|u-v|)f(s,x,v) f(s,x,u)  dv Q(d\theta) d\phi  dx du 
  \end{eqnarray*}
Since 
\begin{equation*}
|\{\psi(x,u+\alpha(u,v,\theta,\phi))-\psi(x,u)\}| \le \sup
 _{\{\eta \in K\}} |\nabla_u \psi(x,\eta)| |\alpha(u,v,\theta,\phi)|
\end{equation*} 
 and 
\begin{equation*}
 |\alpha(u,v,\theta,\phi)|\s(|u-v|)\leq |u-v|^{1+\gamma} |\sin(\frac{\theta}{2})|,
\end{equation*} 
  it follows that  
\begin{eqnarray}
 &|g(s)-g(s_0)|\leq C   \int_{ K \times \R^6}  |f(s,x,u) f(s,x,v)-f(s_0,x,u) f(s_0,x,v)|  \nonumber  \\& \times 
(|u|^{\gamma+1}+|v|^{\gamma +1}) dx du dv,  \label{eqg} 
\end{eqnarray}
with  
\begin{equation*} C:= \|\nabla_u\psi\|_\infty 2  \pi \int_0^\pi \theta Q(d\theta)
\end{equation*}
We split it into two terms, one with $|u|^{\gamma+1}$ (resp. $|v|^{\gamma+1}$), and get 
\begin{eqnarray}
 &\!\!\!\!C \int_{ K \times \R^6} 
 |f(s,x,u) f(s,x,v)-f(s_0,x,u) f(s_0,x,v)||u|^{\gamma+1} dx du dv  \nonumber \\&
 \!\!\!\!\!\!\!\!\!\!\!\!\!\!\!\!= C \int_{ K \times \R^6} 
 |f(s,x,u) f(s,x,v)-f(s,x,u) f(s_0,x,v)||u|^{\gamma+1} dx du dv \notag\\
 &+\,
 C \int_{ K \times \R^6} 
 |f(s,x,u) f(s_0,x,v)-f(s_0,x,u) f(s_0,x,v)||u|^{\gamma+1} dx du dv\,\,\,\,\,\,\,\label {splitwithC}\\&
 =J_1(s)+J_2(s) \nonumber
 \end{eqnarray}
  \begin{eqnarray*}
 & J_1(s)=C\int_{ K \times \R^6} |u|^{\gamma+1}  f(s,x,u) |\int_{s_0}^s \frac{\partial f}{\partial r}(r,x,v)dr |dx du dv
  \\&\leq
  C \left( \sup_{x \in \R^3} \int_{\R^3} |u|^{\gamma+1} f(s,x,u)  du \right) \int_{s_0}^s  \int_{\R^6}|\frac{\partial f}{\partial r}(r,x,v)|dxdvdr
  \end{eqnarray*}
By {\bf B1} and {\bf B2} $\lim_{s\to s_0} J_1(s)=0$.\\
Let us consider $J_2(s)$.
 \begin{eqnarray*}
 & J_2(s)=C\int_{ K \times \R^6} |u|^{\gamma+1}  f(s_0,x,v) |\int_{s_0}^s \frac{\partial f}{\partial r}(r,x,u)dr |dx du dv
  \\&\leq C\int_{s_0}^s \int_{ K \times \R^3} \left( \int_{\R^3}|u|^{\gamma+1} |\frac{\partial f}{\partial r}(r,x,u)|du\right) f(s_0,x,v)  dxdvdr
  \end{eqnarray*}
  Since $\sup_{x\in \R^3}\int_{R^3}|u|^{\gamma+1} |\frac{\partial f}{\partial r}(r,x,u)|du$ is integrable in $[s_0,s]$ by {\bf B3}, 
  we obtain $\lim_{s\to s_0} J_2(s)=0$.\\
  Likewise, and without any changes in the arguments, it follows
  \begin{equation*}
   \lim_{s\to s_0} C \int_{ F \times \R^6} 
 |f(s,x,u) f(s,x,v)-f(s_0,x,u) f(s_0,x,v)||v|^{\gamma+1} dx du dv = 0
  \end{equation*}
Hence $\lim_{s\to s_0} g(s)= g(s_0)$, so that $g$ is a continuous function and 
\begin{equation*}
   \lim_{\Delta t \downarrow 0} \frac{1}{\Delta t} \int_t^{t+\Delta t} g(s) ds=g(t).
  \end{equation*}
  Note that in the above arguments we have taken $s>s_0$ for simplicity. One may also tke $s_0>s$.
 \end {proof}
  Theorem \ref{Ito} motivates the following definition.
  \begin{dfn}
    Let $\{f(t,x,v)\}_{t \in \R_+^0}$ be {\color{blue} {any}}  collection of densities satisfying Hypothesis {\bf {B}}. 
      Suppose that for  any fixed $T>0$, there exists a stochastic basis 
      $(\Omega,\mathscr{F},(\mathscr{F}_t)_{t\in [0,T]},\mathrm{P})$ and   an adapted process $(X_t,Z_t)_{t\in [0,T]}$ with values on  
      $\mathbb{D}\times \mathbb{D}$ such that
      \begin{itemize}
   \item[\bf{i)}] $(X_t,Z_t)_{t\in [0,T]}$  has  time marginals with density $f(t,x,u)$, for $t \in [0,T]$,
   \item[\bf{ii)}] $(X_t,Z_t)_{t\in [0,T]}$  is a weak solution of  the 
      McKean -Vlasov SDE (\eqref{eq-spaceB} \eqref{eq-velB}).
   \end{itemize} 
      Then we say that the    
      McKean -Vlasov equation  (\eqref{eq-spaceB}, \eqref{eq-velB}) with density functions 
$\{f(t,x,v)\}_{t \in \R}$ 
      is associated to the  the 
      Boltzmann equation'' \eqref{eqn1.1}.\\
      If the above property holds for $T\in [0,S]$ , with $S>0$ then  the    
      McKean -Vlasov SDE (\eqref{eq-spaceB}, \eqref{eq-velB})  with density functions $\{f(t,x,v)\}_{t \in [0,S]}$ 
      is associated to the 
      Boltzmann equation \eqref{eqn1.1} up to time $S$.
       \end{dfn}
      
      \begin{Remark}
       Let us assume hypothesis {\bf A}. 
       From  Theorem \ref{Ito} it follows that any stochastic process $(X_t,Z_t)_{t\in [0,T]}$    
    solving a  McKean -Vlasov equation (\eqref{eq-spaceB}, \eqref{eq-velB})   ``associated to the 
      Boltzmann equation \eqref{WBE} is (according to Definition \ref{Denskog}) a Boltzmann process.   
      The Boltzmann equation  \eqref{WBE} is hence  the Fokker-Planck  equation associated to the McKean -Vlasov equation  (\eqref{eq-spaceB}, \eqref{eq-velB}).
      \end{Remark}

 \section{Existence of  Boltzmann processes} \label{Sec Existence Boltzmann process}
 In this Section we assume that  for $B(z, dv, d\theta)$, defined in \eqref{eqn1.8}, the following Hypotheses {\bf A$^{\prime}$} hold:
  \begin{itemize}
   \item[\bf{A1.}] The measure $Q$ is finite outside any neighbourhood of $0$, and for all $\epsilon > 0$, it satisfies 
   \[
    \int_0^{\epsilon} \theta Q(d\theta) < \infty.
   \]
   \item[\bf{A2$^{\prime}$.}] $\sigma: \R_0^+ \to \R_0^+$ (entering \eqref{eqn1.8}) is given by $\sigma(z):=c |z|^{\gamma}$, with $c > 0$, $\gamma =1$ fixed.
   \end{itemize}
Sometimes, we write $\sigma(z, v)$ to denote $\sigma(|z - v|)$. Assuming Hypotheses {\bf A$^{\prime}$}, we prove the existence of Boltzmann processes for the (non-cut-off) case for hard spheres. The case of molecules with soft or hard potentials will be treated in a separate article.\\

 In Theorem \ref{Ito} we proved that any process $(X_t,Z_t)_{t\in [0,T]}$ with time marginal densities $\{f(t,x,z)\}_{t \in [0,T]}$, solving  
  the McKean-Vlasov equation  (\eqref{eq-spaceB}, \eqref{eq-velB}) 
  is a Boltzmann process.\\
The main theorem of  this section  is Theorem \ref{mainThexistenceBoltzmann} in subsection \ref{Par Construction BP}, where, 
given a 
strong  solution $\{f(t,x,z)\}_{t \in [0,T]}$ 
      of the Boltzmann equation  \eqref{eqn1.1}, we find sufficient conditions for the  existence of a solution of the   McKean -Vlasov equation (\eqref{eq-spaceB}, \eqref{eq-velB}) with marginal densities   $\{f(t,x,z)\}_{t \in [0,T]}$. The solution process $\{X_t,Z_t\}_{t\in [0,T]}$ is then a  Boltzmann process. \\
Theorem \ref{thm-uniqueness IWlinBE-given-f} suggests that  it would be convenient  to first prove the existence of a solution $\{X_t,Z_t\}_{t\in [0,T]}$ of a stochastic differential equation 
 associated to the bilinear Boltzmann equation defined in Definition \ref{DeflinBE}. Indeed, this is established in  Theorem \ref{ThCor MR}. In subsection \ref{Par Construction linear BP} we assume  only  that $\{f(t,x,z)\}_{t \in [0,T]}$, appearing on the right side of the SDE (\eqref{eq-spacer}, \eqref{eq-velr}), is any collection of probability densities satisfying some regularity conditions listed in the Assumptions B4 -B6 and prove the existence of a weak solution  having finite second moment (see \eqref{finitesecondmomentsol}).  \\
In subsection \ref{Par Construction BP}  we add the regularity conditions {\bf B0} - {\bf B3} and the  assumption  that $\{f(t,x,z)\}_{t \in [0,T]}$ solves the  Boltzmann equation \eqref{eqn1.1}. We state in Theorem  \ref{ThSDEbilinearBE} that in this case  the SDE  (\eqref{eq-spacer}, \eqref{eq-velr})  is associated to the bilinear Boltzmann equation
   and show in Theorem \ref{thm-existence-given-f}  that,  if the time marginals of the solution $\{X_t,Z_t\}_{t\in [0,T]}$  have  a density, then the densities  coincide with   $\{f(t,x,z)\}_{t \in [0,T]}$.  The proof relies on the same symmetry and entropy   arguments used  in Theorem \ref{thm-uniqueness IWlinBE-given-f}. As a consequence    the SDE  (\eqref{eq-spacer}, \eqref{eq-velr}) coincides  with the McKean -Vlasov SDE  (\eqref{eq-spaceB}, \eqref{eq-velB}) and  its  solution   $\{X_t,Z_t\}_{t\in [0,T]}$ is  a  Boltzmann process. This is stated in Theorem \ref{mainThexistenceBoltzmann}.\\
      

 \subsection{ Construction of a solution of a  SDE associated to the bilinear Boltzmann equation} \label{Par Construction linear BP}
 
 In this paragraph  we assume that  $\{f(t,x,z)\}_{t\in [0,T]}$ is  any  collection of densities 
which 
 satisfies 
the following  conditions:\\

{\bf B4.} 
\[
\sup_{s\in [0, T], x \in \R^3} \int_{\R^3} f(s,x,v) dv \leq C_T < \infty.
\]\\
{\bf B5.} For every $T>0$ and $K>0$, there exists a constant $C_T^K>0$ such that  
\[
\sup_{s\in [0, T],| x | \leq K } \int_{\R^3} max(1,|v|^2
) | {\nabla_x} f(s,x,v)| dv \leq C_T^K < \infty.
\] \\
{\bf B6.}
\[
\sup_{s\in [0, T], x \in \R^3} \int_{\R^3}|v|^3
 f(s,x,v) dv  \leq c_T < \infty.
\]\\
 Obviously, condition {\bf B6} implies that   
\begin{equation*}
\sup_{s\in [0, T], x \in \R^3} \int_{\R^3}|v|^2
 f(s,x,v) dv  \leq \mathbf{C}_T < \infty.
\end{equation*}
for some positive constant $ \mathbf{C}_T$, which sometimes  will be used in the calculations as well.\\


Given a  stochastic basis $(\Omega,\mathscr{F},(\mathscr{F}_t)_{t\in [0,T]},\mathrm{P})$ which supports  an  adapted Poisson noise ${N}(dv,d\theta, d\phi,dr,ds)$   with compensator $m(s,v) dv $$Q(d\theta)d\phi drds$ (where $m(s, v) = \int_{\mathbb{R}^3} f(s, x, v)dx$) and an $\mathcal{F}_0$ -measurable   random vector  $(X_0,Z_0)$  with values on $ \mathbb{R}^3 \times \mathbb{R}^3$, such that
\begin{equation}
\mathbb{E}[|Z_0|^2]<\infty,\quad  
\mathbb{E}[|X_0|^2]<\infty \,,\label{first moment initial}
\end{equation}
our aim is to find a weak solution  of  the following SDE:

     \begin {align}
X_t& = X_0+\int_0^t Z_s ds,\quad \forall t\in[0,T]  \, \label{eq-spacer}\\
 Z_t &= Z_0 + \int_0^t \int_{U_0\times \R^+_0}\alpha(Z_{s_-}, v, \theta, \phi)1_{[0,\; \s(|Z_{s_-} - v|)f(s,X_s|v)]}(r) 
 d {N} \quad \forall t\in[0,T]\, \label{eq-velr}
\end{align}

\noindent  We remark that the above stochastic equation (\eqref{eq-spacer}, \eqref{eq-velr}) is NOT assumed to be  of McKean -Vlasov type.
The following theorem explains why we are interested in finding its solution  and proving that it has  finite second moment for all $t \in [0, T)$.
\begin{theorem}\label{ThSDEbilinearBE}  
 Let  $T>0$ and $\{f(t,x,v)\}_{t \in [0,T]}$ be a collection of  densities which solves the Boltzmann equation  \eqref{eqn1.1}  and satisfies $f(t,x,v)\in C([0,T]\times \R^6)$ and  hypotheses  
 {\bf{B2}}.\\
Suppose that for  fixed $T>0$ there exists a stochastic basis $(\Omega,\mathscr{F},(\mathscr{F}_t)_{t\in [0,T]},\mathrm{P})$,   
an adapted process $(X_t,Z_t)_{t\in [0,T]}$ with values on  $\mathbb{D}\times \mathbb{D}$, 
which satisfies  (\eqref{eq-spacer}, \eqref{eq-velr}) almost surely.  Suppose  that 
$\sup_{t\in[0,T]} \int_{\mathbb{R}^3} |z|^2 \mu_t(dx,dz) < \infty$ where 
$\mu_t$  denotes the distribution of $(X_t, Z_t)$ for each $t \in [0,T]$. Then $\{\mu_t\}_{t \in [0,T]}$ is a weak  solution of the bilinear Boltzmann equation \eqref{weakbilinearBEmeasures}  in $[0,T]$.
\end{theorem}
\begin{proof}
The proof is similar to the proof of Theorem \ref{Ito}.
\end{proof}

\noindent 
Theorem \ref{ThSDEbilinearBE} states that the Fokker-Planck  equation for the solution of 
(\eqref{eq-spacer}, \eqref{eq-velr}) is the bilinear Boltzmann equation \eqref{weakbilinearBEmeasures} in the case where $\{m(t, v)\}_{t \in [0, T]}$ are the velocity marginals of a solution $\{f(t, x, v)\}_{t \in [0, T]}$ of the Boltzmann equation \eqref{eqn1.1}. 
In the definition given below, we define  the martingale problem on the canonical space associated to (\eqref{eq-spacer}, \eqref{eq-velr})  for any given collection of densities $f(t, x, z)\}_{t \in [0, T]}$ with $m(t, z) = \int_{\mathbb{R}^3} f(t, x, z)dx$. 
 We define with  $\mathbb{D}[0,T]:= \mathbb{D}([0,T],\mathbb{R}^3)$, the  space of all  right continuous functions with left limits defined  on 
$[0, T]$ taking values in $\mathbb{R}^3$. Let $\mathcal{D}_T$$:=\mathcal{B}(\mathbb{D}[0,T])$ be the  Borel-$\sigma$-algebra induced by the Skorohod topology on $\mathbb{D}[0,T]$.
We denote the value of any $\omega:=(\tilde{\omega},\hat{\omega}) $ $\in \mathbb{D}[0,T]\times \mathbb{D}[0,T]$ at any time $t\in [0,T]$ by  $\omega_t:=(\tilde{\omega}_t,\hat{\omega}_t) $.  
Likewise,  the time marginal of a Borel probability measure 
$\mu$ on $\mathcal{D}_T \otimes \mathcal{D}_T $ will be denoted by $\mu_t$ for all $t \in [0, T]$. The measure $\mu_t$ will be a Borel probability measure on  $(\mathbb{R}^3\times \mathbb{R}^3,  \mathcal{B}(\mathbb{R}^3)\otimes\mathcal{B}(\mathbb{R}^3)) $.\\

\begin{dfn}\label{MP} Let $\mathcal{L}(X_0,Z_0)$ denote the law of the random vector  $(X_0, Z_0)$  with values on $ \mathbb{R}^3 \times \mathbb{R}^3$.  Let $T>0$. 
A probability measure $\mu$ on the filtered space $(\mathbb{D}[0,T]\times \mathbb{D}[0,T], \{\mathcal{D}_t \otimes \mathcal{D}_t\}_{t\in [0,T]}, \mathcal{D}_T \otimes \mathcal{D}_T)$ is the solution of the martingale problem posed by  (\eqref{eq-spacer}, \eqref{eq-velr}) and subject to \eqref{first moment initial}, is a solution of the martingale problem for $(\mathcal{L}(X_0,Z_0), \mathcal{L}_f)$, i.e. 
\begin{itemize}
\item[i)] $\mu_0:=\mathcal{L}(X_0,Z_0)$
\item[ii)] $\psi(\tilde{\omega}_t,\hat{\omega}_t) -\psi(\tilde{\omega}_0,\hat{\omega}_o) -\int_0^t \mathcal{ L}_f\psi(\tilde{\omega}_s,\hat{\omega}_s) ds$  is an  $\{\mathcal{D}_t\otimes\mathcal{D}_t\}_{t \in [0,T]}$ -martingale w.r.t. $\mu$ for all $\psi\in C_0^2(\R^3\times\R^3)$, 
\end{itemize}
where 
\begin{equation}\label{generator}
\mathcal{ L}_f\psi(x,z) = (z,\nabla_x\psi(x,z))+ L_{f(t)}\psi(x,z) \quad \forall\, \psi\in C_0^2(\R^3\times\R^3)
\end{equation}
and $ L_{f(t)}$  is defined through \eqref{generatornonlinear}
\end{dfn}

The main results of this subsection are the following Theorem \ref{Theorem MR}, its Corollary  \ref{Corollary MR} and Theorem \ref{ThCor MR}.
\begin{theorem}\label{Theorem MR} 
 Let  $T>0$ and $\{f(t,x,v)\}_{t \in [0,T]}$ be a collection of  densities which satisfy $f(t,x,v)\in C([0,T]\times \R^6)$ and  hypotheses   {\bf{B4}} -{\bf{B6}}. \\
There is a solution $\mu$  on the filtered space $(\mathbb{D}[0,T]\times \mathbb{D}[0,T], \{\mathcal{D}_t \otimes \mathcal{D}_t\}_{t\in [0,T]}, \mathcal{D}_T \otimes \mathcal{D}_T)$ of the  martingale problem posed by  (\eqref{eq-spacer}, \eqref{eq-velr}) and subject to \eqref{first moment initial} up to time $T>0$.\\
Moreover 
\begin{equation}\label{secondmomentmeasure}
\sup_{t\in[0,T]} \int_{\mathbb{R}^3} (|x|^2+|z|^2) \mu_t(dx,dz) < \infty
\end{equation}
\end{theorem}
The proof of Theorem \ref{Theorem MR} is long and  requires several steps which will be presented in the next Subsection \ref{SubPar}. As an immediate corollary of   Theorem \ref{Theorem MR}, we conclude that  a weak solution of (\eqref{eq-spacer}, \eqref{eq-velr}) exists in the following sense.

\begin{dfn} A "weak solution"  of equation (\eqref{eq-spacer}, \eqref{eq-velr}) in the time interval $[0,T]$ is a triplet $((\Omega,\mathscr{F},(\mathscr{F}_t)_{t\in [0,T]},\mathrm{P}), {N}(dv,d\theta, d\phi,dr,ds), $$(X_t,Z_t)_{t\in [0,T]})$ for which the following properties hold 
\begin{itemize}
\item $(\Omega,\mathscr{F},(\mathscr{F}_t)_{t\in [0,T]},\mathrm{P})$ is a stochastic basis;
\item  ${N}(dv,d\theta, d\phi,dr,ds)$ is an adapted  Poisson random measure with compensator $m(s,v) dv Q(d\theta)d\phi  ds dr $;
\item $(X_\cdot,Z_\cdot):=(X_t,Z_t)_{t\in [0,T]})$ is an adapted c\`adl\`ag stochastic process with valued in $\mathbb{R}^d\times\mathbb{R}^d$ which satisfies (\eqref{eq-spacer}, \eqref{eq-velr}) $P$ -a.s.
\end{itemize}
\end{dfn}

\par


\begin{cor}\label{Corollary MR} 
 Let  $T>0$ and $\{f(t,x,v)\}_{t \in [0,T]}$ be a collection of  densities which satisfy $f(t,x,u)\in C([0,T]\times \R^6)$ and  hypotheses   $\bf{B_4} - \bf{B_6}$. \\
 There is a weak solution $((\Omega,\mathscr{F},(\mathscr{F}_t)_{t\in [0,T]},\mathrm{P}), {N}(dv,d\theta, d\phi,dr,ds), $$(X_t,Z_t)_{t\in [0,T]})$ of equation (\eqref{eq-spacer}, \eqref{eq-velr} ) in the time interval $[0,T]$, 
such that
\begin{equation} \label{finitesecondmomentsol}
\sup_{t\in[0,T]} \mathbb{E}[|X_t|^2] +\sup_{t\in[0,T]} \mathbb{E}[|Z_t|^2] < \infty  \end{equation}
\end{cor}

\begin{proof}
  First, note that
\begin{equation*}
 |\int_{\mathbb{R}^3\times \Xi} \alpha(z, v, \theta, \phi) \sigma(|z - v|)f(s, x, v)dvQ(d\theta)d\phi|,\end{equation*}
and 
\begin{equation*}\int_{\mathbb{R}^3\times \Xi} |\alpha(z, v, \theta, \phi)|^2 \sigma(|z - v|)f(s, x, v) dvQ(d\theta)d\phi
\end{equation*}
are bounded on compacts as functions of $(s, x, z)$ by the basic estimate
\eqref{growth} along with the conditions {\bf B4} and {\bf B6}. In addition, for any $\psi \in C_b^1(\mathbb{R}^6)$, we have that 
\begin{align*}
&R_t\psi(x, z) \\&:= \nabla\psi(x, z)\cdot [z + \int_{U_0} \a(z, v, \theta, \phi)\sigma(|z - v|) f(t, x, v)dvQ(d\theta)d\phi] \\
&
+ \int_{U_0} \{\psi(x, z + \a(z, v, \theta, \phi)) - \psi(x, z) - \nabla\psi(x, z) \cdot \a(z, v, \theta, \phi)\} \\&\,\,\,\,\,\,\,\,\,\,\,\,\,\,\,\,\,\,\,\,\,\sigma(|z - v|) f(t, x, v)dvQ(d\theta)d\phi
\end{align*}
is well-defined, and 
\[
|R_t\psi(x, z)| \leq C\, \norm{\psi} \,\{|z| + |z|^2 + \int_{U_0} (|v| + |v|^2)f(t, x, v)dv\} \,< \infty
\]
where $\norm{\psi} = \displaystyle{\sup_{x, z} [|\psi(x, z)| + |\nabla \psi(x, z)|]}$.

From this, it follows that  all the conditions in Jacod \cite{Jbook} as stated in \cite[Theorem A.1]{HK90}  are satisfied, and hence,  the existence of a solution $\mu$  
of the  martingale problem posed by  the law of $(X_0, Z_0)$ and the operator $\mathcal{L}_{f(t)}:
0 \le t \le T$ 
 is equivalent to the existence of  a weak solution $((\Omega,\mathscr{F},(\mathscr{F}_t)_{t\in [0,T]},\mathrm{P}),  $ ${N}(dv,d\theta, d\phi,dr,ds),$ $ (X_t,Z_t)_{t\in [0,T]})$ of equation  (\eqref{eq-spacer},\eqref{eq-velr}) in the time interval $[0,T]$. 
\end{proof}

Before we prove Theorem \ref{Theorem MR}, we show in Lemma \ref{LemmaboundnormS1} given below  that any  weak solution of  equation  (\eqref{eq-spacer},\eqref{eq-velr}) which satisfies \eqref{finitesecondmomentsol} 
 is bounded in norm in the Banach space $S_T$ defined here.
\begin{dfn}
 Let  $(\Omega,\mathscr{F},(\mathscr{F}_t)_{t\in [0,T]},\mathrm{P})$ be a given filtered probability space satisfying the usual conditions.   
Let $S_T : = S_T^1(\mathbb{R}^d)$ denote the linear  space of all adapted c\`adl\`ag processes $(X_t)_{t\in [0,T]}$ 
with values on $\mathbb{R}^d$ equipped with norm 
\begin{equation} \label{norm1}
\|X\|_{S_T^1}:= \mathbb{E}[\sup_{s\in[0,T]}|X_s|].
\end{equation}
$S_T^1(\R^d)$ is a non-separable Banach space. (see, for e.g. the proof of  Lemma 4.2.1  in \cite {MR}).
\end{dfn}
In view of our context, set  $d = 6$.\\

\begin{lemma}\label{LemmaboundnormS1}
 Let  $T>0$ and $\{f(t,x,v)\}_{t \in [0,T]}$ be a collection of  densities which satisfy $f(t,x,u)\in C([0,T]\times \R^6)$ and  hypotheses   $\bf{B_4} - \bf{B_6}$. \\
 For any weak solution $((\Omega,\mathscr{F},(\mathscr{F}_t)_{t\in [0,T]},\mathrm{P}), {N}(dv,d\theta, d\phi,dr,ds), $$(X_t,Z_t)_{t\in [0,T]})$ of equation (\eqref{eq-spacer}, \eqref{eq-velr}) in the time interval $[0,T]$, 
such that \eqref{finitesecondmomentsol} holds $\|Z_\cdot\|_{S^{1}_T}$ is bounded, and there exist positive constants $k_T$ and $m_T$ such that 
\begin{equation}\label{boundStbyL2s}
\|Z_\cdot\|_{S^{1}_T}\leq k_T \sup_{t \in [0,T]}\mathbb{E}[|Z_t|^2] +m_T + \mathbb{E}[|Z_0|].
\end{equation} 
\end{lemma}
\begin{proof}
Let us introduce the notation \begin{equation}\label{stochintmods}
\hat{I}(Z)_{t}:= \int_0^t  \int_{U_0\times \R^+_0} |\alpha(Z_{s_-}, v, \theta, \phi)|1_{[0,\; \s (\mathcal{Z}_{s_-}, v)f(s, \mathcal{X}_s | v)]}(r)
 d{N}
\end{equation}
Then
 \begin{equation}\label{expsupZs}
        \mathbb{E}[ \sup_{t\in [0,T]} |Z_t| ]\leq   \mathbb{E}[\hat{I}({Z})_T] +\mathbb{E}[|{Z}_0|] 
         \end{equation}
Using \eqref{growth},  {\bf{B4}}, and {\bf{B6}}
 it follows 
\begin{align}
&\mathbb{E}[\hat{I}({Z})_{T}] =  
\int_0^T \int_{U_0 \times \R^+_0} \mathbb{E}[  |\alpha({Z}_{s}, v, \theta, \phi)| \s(\mathcal{Z}_{s_{-}}, v)f(s, {X}_s, v)]dv d\phi Q(d\theta)ds \notag
\\& \leq C 
  \left(\int_0^T  \sup_{s'\in [0,s]} \mathbb{E} [   |{Z}_{s'}|^2
] \int_{\mathbb{R}^3} f(s,{X}_s,v)dv]  ds  + \int_0^T \int_{\mathbb{R}^3} \mathbb{E}[|v|^2
f(s, {X}_s, v)]dvds\right) \notag
  \\& \leq C_T C
    \left(\int_0^T \sup_{s\in [0,t]} \mathbb{E} [ |{Z}_s|^{2}] dt +  \int_0^T \sup_{x \in \mathbb{R}^3}\{\int_{\mathbb{R}^3}
 |v|^2 
f(s, x, v)dv\}ds \right)\notag\\&
   \leq C_T C
   \left(\int_0^T \sup_{s\in [0,t]} \mathbb{E} [ |{Z}_s|^{2}] dt +T \mathbf{C}_T \right) \label{ineqIs}
\end{align} 
  
 \eqref{boundStbyL2s} follows then from \eqref{expsupZs} and \eqref{ineqIs}.

\end{proof}

We consolidate our results in this Section and state the following Theorem.
\begin{theorem}\label{ThCor MR}
 Let  $T>0$ and $\{f(t,x,v)\}_{t \in [0,T]}$ be a collection of  densities which satisfy $f(t,x,u)\in C([0,T]\times \R^6)$ and  hypotheses   $\bf{B_4} - \bf{B_6}$. \\
 There is a weak solution $((\Omega,\mathscr{F},(\mathscr{F}_t)_{t\in [0,T]},\mathrm{P}), {N}(dv,d\theta, d\phi,dr,ds), $$(X_t,Z_t)_{t\in [0,T]})$ of equation (\eqref{eq-spacer}, \eqref{eq-velr} ) in the time interval $[0,T]$, 
such that $(X_t,Z_t)_{t\in [0,T]})$ $\in S_T^1$ and inequality \eqref{boundStbyL2s}  and \eqref{finitesecondmomentsol} holds.
\end{theorem}
\begin{proof} The result follows from Theorem \ref{Theorem MR}, Corollary \ref{Corollary MR} and Lemma \ref{LemmaboundnormS1}.
\end{proof}

\subsubsection {Construction of the solution of   (\eqref{eq-spacer},\eqref{eq-velr}) }  \label{SubPar}
This subsection is dedicated to the proof of Theorem \ref{Theorem MR}.
 The  weak solution  of    (\eqref{eq-spacer},\eqref{eq-velr}) , subject to  \eqref{first moment initial}, will  be constructed by taking the limit of  strong solutions of a sequence of modified SDEs, defined below in (\eqref{eq-spacer-j-cor}, \eqref{eq-velr-j-cor}.

Let $j\in \mathbb{N}$, $B_j:=\{z\in \mathbb{R}^3:\, |z|\leq j\}$ and 
\begin{equation}  \label{alpha-loc}\alpha_j(z,v,\theta, \phi):=\alpha(\frac{z}{1+d(z,B_j)},v,\theta, \phi)
\end{equation}
and also 
\begin{equation} \label{sigma-loc}\sigma_j(z,v):=\sigma(|\frac{z}{1+d(z,B_j)}-v|)
\end{equation}
where $d(z,B_j)$ denotes the distance of $z\in \mathbb{R}^3$ from $B_j$. 

\begin{theorem}\label{existence-uniqueness Lipschitz}
Let $f(t,x,u)\in C([0,T]\times \R^6)$, and $\{f(t,x,v)\}_{t \in [0,T]}$ be a collection of  densities which   
satisfies  hypotheses { {\bf{B4}} - {\bf{B6}}}.
 Let $j\in \mathbb{N}$ be fixed. \\
 For any  stochastic basis $(\Omega,\mathscr{F},(\mathscr{F}_t)_{t\in [0,T]},\mathrm{P})$ which supports an adapted Poisson noise $\mathcal{N}(dv,d\theta, d\phi,dr,ds)$  with compensator $m(s,v) dv $$Q(d\theta)d\phi  ds dr $ and any  random vector   $(\mathcal{X}_0,\mathcal{Z}_0)$ on $(\Omega, \mathcal{F}_0,P)$ with values on $ \mathbb{R}^3 \times \mathbb{R}^3$ having finite second  moment, 
there  exists  on the same filtered space an adapted Poisson random measure ${N}^j(dv,d\theta, d\phi,dr,ds)$  with compensator $m(s,v) dv $ $Q(d\theta)d\phi   dr  ds$ and a   stochastic process  $(\mathcal{Z}_\cdot, \mathcal{X}_\cdot) \in $   $S_T^1$,  
 which solves 
  \begin {align}
\mathcal{X}_t & = \mathcal{X}_0+\int_0^t \mathcal{Z}_s ds \quad \forall t\in [0,T].  \label{eq-spacer-j-cor}\\   \mathcal{Z}_t &= \mathcal{Z}_0+ \int_0^t \int_{U}\alpha_j(\mathcal{Z}_{s_-}, v, \theta, \phi)  \notag\\
    &\times 1_{[0,\; \s_j(\mathcal{Z}_{s_-}, v) f(s,\mathcal{X}_s |v)]}(r)
 {N}^j(dv,d\theta, d\phi, dr, ds) \quad \forall t\in [0,T] \label{eq-velr-j-cor}  
  \end{align} 
 $P$ -a.s.
\end{theorem}


\begin{proof}{\it{ of Theorem} \ref{existence-uniqueness Lipschitz}}\\
First we prove that there is a predictable process $\phi_j(s,v)$  with values in $[0,2\pi)$ 
and  an adapted stochastic process  $(\mathcal{X}_\cdot, \mathcal{Z}_\cdot) \in $   $S_T^1$    which satisfies  
  \begin {align}
\mathcal{X}_t& = \mathcal{X}_0+\int_0^t \mathcal{Z}_s ds \quad \forall t\in [0,T].  \label{eq-spacer-j}\\
    \mathcal{Z}_t &= \mathcal{Z}_0+ \int_0^t \int_{U}\alpha_j(\mathcal{Z}_{s_-}, v, \theta, \phi+\phi_j(s,v))  \notag\\
    &\times 1_{[0,\; \s_j(\mathcal{Z}_{s_-}, v) f(s,\mathcal{X}_s |v)]}(r)
 \mathcal{N}(dv,d\theta, d\phi, dr, ds) \quad \forall t\in [0,T] \label{eq-velr-j}
\end{align} 
    $P$ -a.s. \\
 We start by noting that 
\begin{equation}\label{growth normalized}
\frac{|z|}{1+d(z,B_j)}\leq \min(j,|z|)
\end{equation}
and there exists a constant $K_j>0$, such that
\begin{equation}\label{Lipschitz normalized}
\left|\frac{z}{1+d(z,B_j)}-\frac{z'}{1+d(z',B_j)}\right| \leq K_j|z-z'|.
\end{equation}
Assume   $(\mathcal{X},\mathcal{Z})_{t\in [0,T]}\in S_T^1$. 
For any $\phi_j$,  the stochastic integrals 
 \begin{equation}\label{stochintmodj}
\hat{I}_j(\mathcal{Z})_{t}:= \int_0^t  \int_{U_0\times \R^+_0} |\alpha_j(\mathcal{Z}_{s_-}, v, \theta, \phi+\phi_j(s,v))|1_{[0,\; \s_j(\mathcal{Z}_{s_-}, v)f(s, \mathcal{X}_s | v)]}(r)
 d\mathcal{N}
\end{equation}
 and 
   \begin{equation}\label{stochintj}
{I}_j(\mathcal{Z})_{t}:= \int_0^t  \int_{U_0\times \R^+_0} \alpha_j(\mathcal{Z}_{s_-}, v, \theta, \phi+\phi_j(s,v))1_{[0,\; \s_j(\mathcal{Z}_{s_-}, v)f(s, \mathcal{X}_s | v)]}(r)
 d\mathcal{N} 
\end{equation}
are well defined.
In fact, using \eqref{growth},  {\bf{B4}}, {\bf{B6}} it follows 
\begin{align*}
&\mathbb{E}[\hat{I}_j(\mathcal{Z})_{T}] =  
\int_0^T \int_{U_0 } \mathbb{E}[  |\alpha_j(\mathcal{Z}_{s}, v, \theta, \phi+\phi_j(s,v))| \s_j(\mathcal{Z}_{s_{-}} ,v)f(s,\mathcal{ X}_s, v)]dv d\phi Q(d\theta)ds 
\\& \leq C 
  \left( \int_0^T \mathbb{E} [ {\frac{|\mathcal{Z}_s|^2
}{(1+d(\mathcal{Z}_s,B_j))^2
}} \int_{\mathbb{R}^3} f(s,\mathcal{X}_s,v)dv] ds + \int_0^T \int_{\mathbb{R}^3} \mathbb{E}[|v|^2
f(s, \mathcal{X}_s, v)]dvds\right) 
 \\& \leq C_T C
    \left(T j^2
+  \int_0^T \sup_{x \in \mathbb{R}^3}\{\int_{\mathbb{R}^3} |v|^2
f(s, x, v)dv\}ds \right) \,<\infty 
\end{align*} 

It follows  that 
\begin{equation}\label{ineqstochintmodj}
\mathbb{E}[\hat{I_j}(\mathcal{Z}))_{T}] \leq K  j^2
T +M_T\,, 
\end{equation}
with $K>0$, $M_T>0$. 
Let $ (S_j\mathcal{X},S_j\mathcal{Z})_{t\in [0,T]}$ denote the process defined through 
\begin{equation*}
    S_j\mathcal{Z}_t := \mathcal{Z}_0 +I_j(\mathcal{Z})_t \;\;\text{and}\;\;\; 
    S_j\mathcal{X}_t :=\mathcal{X}_0+\int_0^t \mathcal{Z}_s ds\,.   
    \end{equation*}
    If the random vector $(\mathcal{X}_0,\mathcal{Z}_0)$ satisfies \eqref{first moment initial},  then 
       \begin{equation}\label{supS}
    \sup_{t\in [0,T]} |S_j\mathcal{Z}_t| \leq \hat{I}_j(\mathcal{Z})_T +|\mathcal{Z}_0| \quad a.s.
    \end{equation}
    and hence 
     \begin{equation}\label{expsupSj}
        \mathbb{E}[ \sup_{t\in [0,T]} |S_j\mathcal{Z}_t| ]\leq   \mathbb{E}[\hat{I}_j(\mathcal{Z})_T] +\mathbb{E}[|\mathcal{Z}_0|] 
         \end{equation}

             and, with \eqref{ineqstochintmodj}, one concludes that $(S_j\mathcal{X},S_j\mathcal{Z})_{t\in [0,T]}\in S_T^1$. \\
 
 Given a stochastic basis $(\Omega,\mathscr{F},(\mathscr{F}_t)_{t\in [0,T]},\mathrm{P})$ with an adapted Poisson noise $\mathcal{N}(dv,d\theta, d\phi,dr,ds)$  with compensator $m(s,v) dv $$Q(d\theta)d\phi  ds dr $ and a  random vector   $(\mathcal{Z}_0,\mathcal{X}_0)$ on $(\Omega, \mathcal{F}_0,P)$  which  
 satisfies \eqref{first moment initial}, a Picard iteration scheme is constructed which converges in $S_T^1$ to the solution of  (\eqref{eq-velr-j} , \eqref{eq-spacer-j}).
 Following ideas from   Tanaka \cite{Ta} and   Fournier  and M\'el\'eard  \cite{FM01} the iteration    involves the translation of the parameter $\phi$ given in Lemma \ref{rotationTa}. 
For $n\in \mathbb{N}$ we define 
 
   \begin {align}
  \mathcal{X}_t^0 &= \mathcal{X}_0 \notag\\ 
    \mathcal{Z}_t^0 &= \mathcal{Z}_0 \notag\\ 
\mathcal{X}_t^1& = \mathcal{X}_0+\int_0^t \mathcal{Z}_s^1 ds, \notag\\
    \mathcal{Z}_t^1 &= \mathcal{Z}_0 \notag + \int_0^t \int_{U}\alpha_j(\mathcal{Z}^0_{s_-}, v, \theta,\phi) 
    \notag\\
    &\times\ 1_{[0,\; \s_j(\mathcal{Z}^0_{s_-}, v) f(s,\mathcal{X}^0_s |v)]}(r)
 \mathcal{N}(dy,dv, d\theta, d\phi, dr, ds), \notag\\
\mathcal{X}^{n+1}_t& = \mathcal{X}_0+\int_0^t \mathcal{Z}^{n+1}_s ds, , \,\,\,n \ge 1   \label{eq-spacer-j-P}\\
    \mathcal{Z}^{n+1}_t &= \mathcal{Z}_0+ \int_0^t \int_{U}\alpha_j(\mathcal{Z}^n_{s_-}, v, \theta, \phi+ \phi^{n})  \notag\\
    &\times 1_{[0,\; \s_j(\mathcal{Z}^n_{s_-}, v) f(s,\mathcal{X}^n_s |v)]}(r)
 \mathcal{N}(dy,dv,d\theta, d\phi, dr, ds) \label{eq-velr-j-P}.
    \end{align}
    where
    \begin{align*}
    &\phi^1:=\phi_0(\frac{\mathcal{Z}^0_{s_-}}{1+d(\mathcal{Z}_{s_-}^0,B_j)},v,\frac{ \mathcal{Z}^1_{s_-}}{1+d(\mathcal{Z}_{s_-}^1,B_j)},v)\\
   & \phi^{n+1}:=\phi^n+\phi_0(\frac{\mathcal{Z}^n_{s_-}}{1+d(\mathcal{Z}_{s_-}^n,B_j)},v, \frac{\mathcal{Z}^{n+1}_{s_-}}{1+d(\mathcal{Z}_{s_-}^{n+1},B_j)},v)
     \end{align*}
     
      We first prove the following inequality:
    \begin{equation}\label{Lipalphan}
    | \alpha_j(\mathcal{Z}^{n}_{s_-}, v, \theta,\phi+  \phi^{n})- \alpha_j(\mathcal{Z}^{n-1}_{s_-}, v, \theta,\phi+  \phi^{n-1})|\leq 2 K_j \theta|\mathcal{Z}^{n}_{s_-}-\mathcal{Z}^{n-1}_{s_-}| 
    \end{equation}
    In fact, using \eqref{Lipschitz normalized} and \eqref{PrameterTrans}, it follows 
    \begin{align*}
    &| \alpha_j(\mathcal{Z}^{n}_{s_-}, v, \theta, \phi+\phi^{n})- \alpha_j(\mathcal{Z}^{n-1}_{s_-}, v, \theta, \phi +\phi^{n-1})|\leq 
    \\&
    K_j \sin^2(\frac{\theta}{2})   |\mathcal{Z}^{n}_{s_-}-\mathcal{Z}^{n-1}_{s_-}| +  |\frac{\sin(\theta)}{2}|\times \notag\\&|\Gamma(v-\frac{\mathcal{Z}^{n}_{s_-}}{1+d(\mathcal{Z}^{n}_{s_-}, B_j)},\phi+\phi^{n-1}+\phi_0(\frac{\mathcal{Z}^{n-1}_{s_-}}{1+d(\mathcal{Z}^{n-1}_{s_-}, B_j)},v, \frac{\mathcal{Z}^{n}_{s_-}}{1+d(\mathcal{Z}^{n}_{s_-}, B_j)},v) )\notag\\&- \Gamma(v-\frac{\mathcal{Z}^{n-1}_{s_-}}{1+d(\mathcal{Z}^{n-1}_{s_-}, B_j)} ,\phi+\phi^{n-1})|
\end{align*}
   
with 
\begin{align}&|\Gamma(v-\frac{\mathcal{Z}^{n}_{s_-}}{1+d(\mathcal{Z}^{n}_{s_-}, B_j)},\phi+\phi^{n-1}+\phi_0(\frac{\mathcal{Z}^{n-1}_{s_-}}{1+d(\mathcal{Z}^{n-1}_{s_-}, B_j)},v, \frac{\mathcal{Z}^{n}_{s_-}}{1+d(\mathcal{Z}^{n}_{s_-}, B_j)},v) )\notag\\&- \Gamma(v-\frac{\mathcal{Z}^{n-1}_{s_-}}{1+d(\mathcal{Z}^{n-1}_{s_-}, B_j)} ,\phi+\phi^{n-1})|
=  \notag\\& |\Gamma(v-\frac{\mathcal{Z}^{n}_{s_-}}{1+d(\mathcal{Z}^{n}_{s_-}, B_j)},\phi+\phi_0(\frac{\mathcal{Z}^{n-1}_{s_-}}{1+d(\mathcal{Z}^{n-1}_{s_-}, B_j)},v, \frac{\mathcal{Z}^{n}_{s_-}}{1+d(\mathcal{Z}^{n}_{s_-}, B_j)},v) )\notag\\&- \Gamma(v-\frac{\mathcal{Z}^{n-1}_{s_-}}{1+d(\mathcal{Z}^{n-1}_{s_-}, B_j)} ,\phi|\label {cutrotation}\\& \leq 3K_j  |\mathcal{Z}^{n}_s-\mathcal{Z}^{n-1}_s| \label{LipGamma}
    \end{align}
where the last inequality was a consequence of \eqref{LipschitzGamma} and \eqref{Lipschitz normalized} .
It follows that \eqref{Lipalphan}  holds. \\
  
  In order to prove that the iteration $(\mathcal{X}^n_\cdot, \mathcal{Z}^n_\cdot)$ converges in $S_T^1$ we develop some inequalities.

     \begin{align}
&\mathbb{E}[\sup_{s\in [0,t]} |\mathcal{Z}^{n+1}_{s_-}-\mathcal{Z}^n_{s_-}| ]\leq \label{ineqMRL1}\\&
\mathbb{E}[\sup_{s\in [0,t]} | \int_0^s  \int_{U}  \alpha_j(\mathcal{Z}^{n}_{s'_-}, v, \theta,\phi+  \phi^{n})
1_{[0,\; \sigma_j(|\mathcal{Z}^{n}_{s'_-} - v|)f(s',\mathcal{X}^{n}_{s'}| v)]}(r) \notag \\&-\alpha_j(\mathcal{Z}^{n-1}_{s'_-}, v, \theta, \phi+ \phi^{n-1})  \notag
1_{[0,\; \sigma_j(|\mathcal{Z}^{n-1}_{s'_-} - v|)f(s',\mathcal{X}^{n-1}_{s'}| v)]}(r) d\mathcal{N} |] \notag
\\&\leq \mathbb{E}[  \int_0^t  \int_{\R^3}\int_{\mathbb{R}_+^0} \int_0^\pi  \int_0^{2 \pi} |\alpha_j(\mathcal{Z}^{n}_{s'}, v, \theta, \phi+ \phi^{n})
1_{[0,\; \sigma_j(|\mathcal{Z}^{n}_{s'} - v|)f(s',\mathcal{X}^{n}_{s'}| v)]}(r) \notag \\& -\alpha_j(\mathcal{Z}^{n-1}_{s'}, v, \theta, \phi+ \phi^{n-1})
1_{[0,\; \sigma_j(|\mathcal{Z}^{n-1}_{s'} - v|)f(s',\mathcal{X}^{n-1}_{s'} v)]}(r) | d\phi   m(s' v) dv Q(d\theta) dr ds'] \notag
\end{align}
 (For the last inequality, see e.g. the proof of  Lemma 2.4.12 in \cite{MR} or Chapter II, Section 3, page 62 \cite{IW}.) 
We will prove the following inequality
     \begin{align}\label{contractionforalphaj}
       & \int_0^t  \int_{\R^3} \int_{\mathbb{R}_+^0}\int_0^\pi \int_0^{2 \pi} | \alpha_j(\mathcal{Z}^{n}_{s_-}, v, \theta,\phi+  \phi^{n})1_{[0,\; \s_j(\mathcal{Z}^{n}_{s_-} , v) f(s, \mathcal{X}^{n}_s |v) ]}(r)\notag\\& - \alpha_j(\mathcal{Z}^{n-1}_{s_-}, v, \theta,\phi+ \phi^{n-1})
    1_{[0,\; \s_j(\mathcal{Z}^{n-1}_{s_-} , v)f(s, \mathcal{X}^{n-1}_s| v)]}(r)| d\phi  m(s,v)dv Q(d\theta) dr ds 
  \notag  \\&\leq 
C_{j} \int_0^t (|\mathcal{X}^{n}_s-\mathcal{X}_s^{n-1}|+ 2 |\mathcal{Z}^{n}_{s_-}-\mathcal{Z}_{s_-}^{n-1}|)ds\quad P\,-a.s.
     \end{align}
From \eqref{ineqMRL1} and  \eqref{contractionforalphaj} it follows that
\begin{align} \label{contraction}
&\mathbb{E}[\sup_{s\in [0,t]} |\mathcal{Z}^{n+1}_{s_-}-\mathcal{Z}^n_{s_-}| ]+ \mathbb{E}[\sup_{s\in [0,t]} |\mathcal{X}^{n+1}_s-\mathcal{X}^n_s| ] \notag\\& \leq 
 \int_0^t \mathbb{E}[(2 C_{j}
+1)|\mathcal{Z}^{n}_{s'}-\mathcal{Z}^{n-1}_{s'}|+   C_{j} 
|\mathcal{X}^{n}_{s'}-\mathcal{X}^{n-1}_{s'}| ]ds' \notag
\\&\leq (2 C_{j}
  +1) \int_0^t \mathbb{E}[\sup_{s'\in [0,s]} \{|\mathcal{Z}^{n}_{s'}-\mathcal{Z}^{n-1}_{s'}|+ |\mathcal{X}^{n}_{s'}-\mathcal{X}^{n-1}_{s'}|\}] ds
\end{align}
It would then follow that there exists a stochastic process $(\mathcal{X}_\cdot, \mathcal{Z}_\cdot) \in S_T^1$ such that 
\begin{equation} \label{convergencePicard}
\lim_{n \to \infty}\|( \mathcal{X}^{n}_\cdot,\mathcal{Z}^{n}_\cdot) -(\mathcal{X}_\cdot,\mathcal{Z}_\cdot)\|_{S_T^1}=0
\end{equation}
hence, we first prove the inequality \eqref{contractionforalphaj}. After this, we show that $(\mathcal{X}_\cdot,\mathcal{Z}_\cdot )$ is a  solution of (\eqref{eq-spacer-j}, \eqref{eq-velr-j}).
 
{\it Proof of \eqref{contractionforalphaj}:}\\
  \begin{align*} & \int_0^t  \int_{\R^3} \int_{\mathbb{R}_+^0}\int_0^\pi \int_0^{2 \pi} | \alpha_j(\mathcal{Z}^{n}_{s_-}, v, \theta, \phi+ \phi^{n})1_{[0,\; \s_j(\mathcal{Z}^{n}_{s_-} , v) f(s, \mathcal{X}^{n}_s |v) ]}(r)\notag\\& - \alpha_j(\mathcal{Z}^{n-1}_{s_-}, v, \theta,\phi+ \phi^{n-1})
    1_{[0,\; \s_j(\mathcal{Z}^{n-1}_{s_-} , v)f(s, \mathcal{X}^{n-1}_s| v)]}(r)| d\phi  m(s,v)dv Q(d\theta) dr ds \notag \\
    \leq & I +II
  \end{align*}
  where  
   \begin{align*}
         & I:=\int_0^t  \int_{U}\left| \alpha_j(\mathcal{Z}^{n}_{s_-}, v, \theta, \phi+ \phi^{n})- \alpha_j(\mathcal{Z}^{n-1}_{s_-}, v, \theta, \phi+ \phi^{n-1}) \right|\\& \times 1_{[0,\; \s_j(\mathcal{Z}^{n}_{s_-} , v) f(s, \mathcal{X}^{n}_s |v) ]}(r)
         d\phi m(s,v)dv Q(d\theta) dr ds\\& \leq  \int_0^t C_j |\mathcal{Z}^{n}_{s_-}-\mathcal{Z}^{n-1}_{s_-}|ds \quad P\,-a.s., 
          \end{align*}
         for some constant $C_j>0$,  where the last  inequality follows from  \eqref{Lipalphan},   \eqref{growth normalized} and condition {\bf B4}.
         \begin{align*}
         & II := \int_0^t  \int_{U}  |\alpha_j(\mathcal{Z}^{n-1}_{s_-}, v, \theta, \phi+ \phi^{n-1})|\times \\&
          |1_{[0,\; \s_j(\mathcal{Z}^{n}_{s_-} , v) f(s, \mathcal{X}^{n}_s |v) ]}(r)- 1_{[0,\; \s_j(\mathcal{Z}^{n-1}_{s_-} , v)f(s, \mathcal{X}^{n-1}_s| v)]}(r)|d\phi m(s,v)dv Q(d\theta) dr ds
          \\& \leq \int_0^t  \int_{U_0} |\alpha_j(\mathcal{Z}^{n-1}_{s_-}, v, \theta, \phi+ \phi^{n-1})|\\& \times \{ \max(\s_j(\mathcal{Z}^{n}_{s_-} , v) f(s, \mathcal{X}^{n}_s |v) ,\s_j(\mathcal{Z}^{n-1}_{s_-} , v)f(s, \mathcal{X}^{n-1}_s| v))\\&
          -  \min(\s_j(\mathcal{Z}^{n}_{s_-} , v) f(s, \mathcal{X}^{n}_s |v) ,\s_j(\mathcal{Z}^{n-1}_{s_-} , v)f(s, \mathcal{X}^{n-1}_s| v))\} d\phi m(s,v)dv Q(d\theta) ds \\&= \int_0^t  \int_{U_0} |\alpha_j(\mathcal{Z}^{n-1}_{s_-}, v, \theta, \phi+ \phi^{n-1})|
          \\& \times \{ \max(\s_j(\mathcal{Z}^{n}_{s_-} , v) f(s, \mathcal{X}^{n}_s ,v) ,\s_j(\mathcal{Z}^{n-1}_{s_-} , v)f(s, \mathcal{X}^{n-1}_s, v))\\&
          -  \min(\s_j(\mathcal{Z}^{n}_{s_-} , v) f(s, \mathcal{X}^{n}_s , v) ,\s_j(\mathcal{Z}^{n-1}_{s_-} , v)f(s, \mathcal{X}^{n-1}_s,  v))\} d\phi dv Q(d\theta) ds 
          \\&= \int_0^t  \int_{U_0} |\alpha_j(\mathcal{Z}^{n-1}_{s_-}, v, \theta, \phi+ \phi^{n-1})|\\& \times  |\s_j(\mathcal{Z}^{n}_{s_-} , v) f(s, \mathcal{X}^{n}_s, v) -\s_j(\mathcal{Z}^{n-1}_{s_-},  v)f(s, \mathcal{X}^{n-1}_s, v)| d\phi  dv Q(d\theta)  ds
    \\& \leq A+B \quad P\,-a.s.,         \end{align*}
    \begin{align*}
    A:=&\int_0^t  \int_{U_0} |\alpha_j(\mathcal{Z}^{n-1}_{s_-}, v, \theta, \phi+ \phi^{n-1})|\\& \times  |\s_j(\mathcal{Z}^{n}_{s_-} , v)| | f(s, \mathcal{X}^{n}_{s_-}, v) -f(s, \mathcal{X}^{n-1}_{s_-}, v)| d\phi  dv Q(d\theta)  ds\\& \leq C \int_0^t  \int_{\R^3} (|\frac{\mathcal{Z}^{n-1}_{s_-}}{(1+d(\mathcal{Z}^{n-1}_{s_-},B_j))} |+ |v| )|\frac{\mathcal{Z}^{n}_{s_-}}{(1+d(\mathcal{Z}^{n}_{s_-},B_j))}-v|
\\&\times | f(s, \mathcal{X}^{n}_{s_-}, v) -f(s, \mathcal{X}^{n-1}_{s_-}, v)| dv ds
    \\& \leq C_j \int_0^t |\mathcal{X}^{n}_{s_-} -\mathcal{X}^{n-1}_{s_-}| ds \quad P\,-a.s.,  
    \end{align*}
    for some constant $C_j>0$, where we used \eqref{growth normalized} as well as condition {\bf B5}.
       \begin{align*}
    B:=&\int_0^t  \int_{U_0} |\alpha_j(\mathcal{Z}^{n-1}_{s_-}, v, \theta, \phi+ \phi^{n-1})|\\& \times 
    |\s_j(\mathcal{Z}^{n}_{s_-} , v)- \s_j(\mathcal{Z}^{n-1}_s , v)|f(s, \mathcal{X}^{n-1}_s,  v))d\phi  dv Q(d\theta)  ds 
    \\&  \leq C  \int_0^t  \int_{\R^3} (|\frac{\mathcal{Z}^{n-1}_{s_-}}{(1+d(\mathcal{Z}^{n-1}_{s_-},B_j))} |+ |v| ) \\&\times |\frac{\mathcal{Z}^{n}_{s_-}}{(1+d(\mathcal{Z}^{n}_{s_-},B_j))} - \frac{\mathcal{Z}^{n-1}_{s_-}}{(1+d(\mathcal{Z}^{n-1}_{s_-},B_j))}|
    f(s, \mathcal{X}^{n-1}_{s_-}, v) dv ds
  \\& \leq C_j \int_0^t |\mathcal{Z}^{n}_{s_-} -\mathcal{Z}^{n-1}_{s_-}|
ds
    \end{align*}
      for some constants $C_j>0$, 
 where we used \eqref{Lipschitz normalized}  as well as conditions {\bf B4} and {\bf B6}.
 Thus \eqref{contractionforalphaj} is proven.


\noindent Let us prove that there is a predictable function $\phi_j(s,v)$ on  $(\Omega,\mathscr{F},(\mathscr{F}_t)_{t\in [0,T]},\mathrm{P})$ such that the stochastic process $(\mathcal{X}_\cdot, \mathcal{Z}_\cdot) \in S_T^1$, obtained through the Picard- iteration in   \eqref{convergencePicard}, solves \eqref{eq-spacer-j} and  \eqref{eq-velr-j}.  We use the same arguments as in the proof of Lemma 4.8 \cite{FM01}, page 593, where the stochastic process associated to the spatially homogenous Boltzmann equation with Maxwellian interaction ($\sigma \equiv 1$) is constructed. We remark however that, since here $\{f(t,x,z)\}_{t\in [0,T]}$ is given, we obtain a strong solution on   $(\Omega,\mathscr{F},(\mathscr{F}_t)_{t\in [0,T]},\mathrm{P})$.\\
From \eqref{cutrotation} and \eqref{LipGamma} and the definition of $\phi_n$,  it follows that
\begin{align*}
&|\Gamma(v-\frac{\mathcal{Z}^{n}_{s_-}}{1+d(\mathcal{Z}^{n}_{s_-}, B_j)},\phi+\phi^{n}) - \Gamma(v-\frac{\mathcal{Z}^{n-1}_{s_-}}{1+d(\mathcal{Z}^{n-1}_{s_-}, B_j)} ,\phi+\phi^{n-1})| \\& \leq 3 K_j |\mathcal{Z}^{n}_s-\mathcal{Z}^{n-1}_s| 
\end{align*}

As a consequence there exists a  predictable function $\delta_j(s,\phi, v)$ such that for $n \to \infty$ 
\begin{equation}\label{limitGamma}
\|\Gamma(v-\frac{\mathcal{Z}^{n}_{s_-}}{1+d(\mathcal{Z}^{n}_{s_-}, B_j)},\phi+\phi^n)-\delta_j(s,\phi, v)\|_{S_T^1} \, \to \, 0 \quad P -a.s.
\end{equation}
with 
$|\delta_j(s,\phi, v)|=|v-\frac{\mathcal{Z}_{s_-}}{1+d(\mathcal{Z}_{s_-}, B_j)}|$ and such that   $v-\frac{\mathcal{Z}_{s_-}}{1+d(\mathcal{Z}_{s_-}, B_j)}$ is orthogonal to $\delta_j(s,\phi, v)$. 
Hence, in particular, these two properties hold for $\delta_j(s,0, v)$ which allow us to conclude that for each $v \in \mathbb{R}^3$, there exists an adapted process $\phi_j(s,v)$  such that 
\begin{equation*}
\delta_j(s,0, v)=\Gamma(v-\frac{\mathcal{Z}^n_{s_-}}{1+d(\mathcal{Z}^n_{s_-}, B_j)},\phi_j(s,v) ).
\end{equation*}
From the above equality, it follows that $\phi_j$  is predictable by using the definition of $\Gamma(v-\frac{\mathcal{Z}^n_{s_-}}{1+d(\mathcal{Z}^n_{s_-}, B_j)},\phi_j(s,v) )$ and predictability of $\delta_j(s,0, v)$.

Using this it follows that 
\begin{equation}\label{defphij}
\delta_j(s,\phi, v)=\Gamma(v-\frac{\mathcal{Z}_{s_-}}{1+d(\mathcal{Z}_{s_-}, B_j)},\phi+\phi_j(s,v) )\quad  \forall \phi \in [0,2\pi) 
\end{equation}
and the stochastic process $(\mathcal{X}_\cdot, \mathcal{Z}_\cdot) \in S_T^1$, obtained through the Picard- iteration in   \eqref{convergencePicard}, solves (\eqref{eq-spacer-j},\eqref{eq-velr-j}) . 


    
  

In the following, we rewrite the SDE (\eqref{eq-spacer-j},\eqref{eq-velr-j}) in the form  (\eqref{eq-spacer-j-cor},\eqref{eq-velr-j-cor}), i.e we prove that a strong solution of  (\eqref{eq-spacer-j},\eqref{eq-velr-j}) is also a strong solution of  (\eqref{eq-spacer-j-cor},\eqref{eq-velr-j-cor}).\\
 In showing this, we follow  a trick by Tanaka \cite{Ta}, which is also used and explained by Fournier and   M\'el\'eard in  \cite{FM}  Lemma 4.7: Given  the Poisson noise $\mathcal{N}(dv,d\theta, d\phi,dr,ds)$ on the stochastic basis $(\Omega,\mathscr{F},(\mathscr{F}_t)_{t\in [0,T]},\mathrm{P})$, defined in Theorem \ref{existence-uniqueness Lipschitz}, we  define a new counting measure $ {N}^j(dv,d\theta, d\phi, dr, ds) $,  obtained  by making a rotation with angle $\phi_j(s,v)$.
\begin{align}\label{rotatedPoisson}
N^j(A)=&\int_0^T \int_U 1_{A}(v,\theta,\phi+ \phi_j(s,v),r,s) \mathcal{N}(dv,d\theta, d\phi,dr,ds)  \\&\quad \quad \quad \quad \quad  \forall A\in \mathcal{B}([0,T])\otimes \mathcal{B}(U_0)\otimes \mathcal{B}(\mathbb{R}_0^+)\notag 
\end{align}
$N^j(dv,d\theta, d\phi,dr,ds)$ is an adapted Poisson random measure on the same stochastic basis  $(\Omega,\mathscr{F},(\mathscr{F}_t)_{t\in [0,T]},\mathrm{P})$, which has the same compensator $m(s,v) dv $$Q(d\theta)d\phi  ds dr $ as $\mathcal{N}(dv,d\theta, d\phi,dr,ds)$. Furthermore, for any predictable and  integrable function $h(v, \theta, \phi, r,s, \omega)$
\begin{align*}
&\int_0^t \int_U h(v,\theta,\phi+ \phi_j(s,v),r,s) \mathcal{N}(dv,d\theta, d\phi,dr,ds) = \\&\int_0^t \int_U h(v,\theta,\phi,r,s) {N}^j(dv,d\theta, d\phi,dr,ds)  \quad \forall t\in [0,T]
\end{align*} 
SDE (\eqref{eq-spacer-j},\eqref{eq-velr-j}) can hence be rewritten in the form  (\eqref{eq-spacer-j-cor},\eqref{eq-velr-j-cor}), i.e a strong solution of  (\eqref{eq-spacer-j},\eqref{eq-velr-j}) is also a strong solution of  (\eqref{eq-spacer-j-cor},\eqref{eq-velr-j-cor}) .
\end{proof}

   \begin{theorem}\label{Th moment control} Let $(\mathcal{X}_\cdot,\mathcal{Z}_\cdot) $ be a solution of   (\eqref{eq-spacer-j-cor}, \eqref{eq-velr-j-cor}).  There exist  constants $K_T\geq 0$,  
     $M_T\geq 0$, $k_T>0$ and $m_T>0$ which do not depend on $j$ and such that the following inequalities holds
     \begin{equation}\label{GrownwallZj}
                        \sup_{t\in [0,T]} \mathbb{E}[  |\mathcal{Z}_t|^2 ]\leq (M_T  +\mathbb{E}[|\mathcal{Z}_0|^2]) e^{T K_T}  
                          \end{equation}  
                          \begin{equation}\label{boundStbyL2j}
\|\mathcal{Z}_\cdot\|_{S^{1}_T}\leq k_T \int_0^T \sup_{t \in [0,T]}\mathbb{E}[|\mathcal{Z}_t|^2]dt +m_T + \mathbb{E}[|\mathcal{Z}_0|].
\end{equation}
From \eqref{GrownwallZj} and \eqref{boundStbyL2j} it follows 
\begin{equation}\label{growthinST1}
\|\mathcal{Z}_\cdot\|_{S^{1}_T}\leq T k_T (M_T  + \mathbb{E}[|\mathcal{Z}_0|^2]) e^{T K_T}  + m_T + \mathbb{E}[|\mathcal{Z}_0|].
\end{equation}

    
\end{theorem}
\begin{proof}
 We will use the symmetry properties of the in-coming and outgoing velocities defined in  \eqref{eqn1.4} to prove 
\begin{equation}\label{expsupZjsquare}
          \sup_{t\in [0,T]}  \mathbb{E}[  |\mathcal{Z}_t|^2 ]\leq K_T \int_0^T \sup_{s\in [0,t]} \mathbb{E} [ |\mathcal{Z}_s|^2] dt +M_T  +\mathbb{E}[|\mathcal{Z}_0|^2] 
             \end{equation}
 with $K_T\geq 0$ and $M_T \geq 0$ not dependent on $j$. Inequality    \eqref{GrownwallZj} would then follow by the Gronwall Lemma.\\
 The proof of \eqref{expsupZjsquare} follows in the spirit of the proof of Lemma \ref{Lemmaconserved}.
 
 {\it Proof of \eqref{expsupZjsquare}:}\\
 Given $(z,v) \in \mathbb{R}^3\times \mathbb{R}^3$, let us denote for simplicity of the notation 
 \begin{equation} 
\left\{
\begin{aligned}
z^{j, \star}=z+\alpha_j(z,v,\theta, \phi) \\
v^{j, \star}=v-\alpha_j(z,v,\theta, \phi), 
\end{aligned}
\right. 
\end{equation}
Where we use the notation $\alpha_j:=\alpha_j(z,v,\theta, \phi)$ and $d_j:=d(z,B_j)$. Then 
\begin{eqnarray} \label{notmoresymm}
&|z^{j, \star}|^2+|v^{j, \star}|^2 =|z+\alpha_j|^2 + |v-  \alpha_j|^2 
=|z|^2 + |v|^2+E_j\\
& \text{with} \quad  E_j= 2 |\alpha_j|^2- 2(v-z, \alpha_j)\notag
\end{eqnarray}
\begin{eqnarray*}
&E_j=  2 |\alpha_j|^2 - 2(v-\frac{z}{1+d_j}, \alpha_j)- 2(\frac{z}{1+d_j}- z, \alpha_j)\\&= 2 |\alpha_j|^2- 2 |\alpha_j|^2- 2(\frac{z}{1+d_j}- z, \alpha_j)  =\frac{2d_j}{1+d_j}(z, \alpha_j)
\end{eqnarray*}
where we have used the definition of $\alpha_j$ in \eqref{alpha-loc} with \eqref{alpha}.
Thus we write equation \eqref{notmoresymm} as
\begin{equation} \label{diff}
|z^{j, \star}|^2 -|z|^2= |v|^2 -|v^{j, \star}|^2 +E_j
\end{equation}
with
\begin{equation}\label{Ej}
E_j:=E_j(z,v):=\frac{2d(z,B_j)}{1+d(z,B_j)}(z, \alpha_j(z,v,\theta, \phi))
\end{equation}

 Let us fix $j \in \mathbb{N}$ and for simplicity of notation  denote with 
 $(\mathcal{ X}_\cdot,\mathcal{Z}_\cdot)$ a solution (\eqref{eq-spacer-j-cor}, \eqref{eq-velr-j-cor}).

 According to \cite{MRT} we can apply the It\^o formula for $\mathcal{Z}_\cdot$ to the function     $\psi(z):=|z|^2$  and obtain 
\begin{eqnarray*}
 &|\mathcal{Z}_t|^2 - |\mathcal{Z}_0|^2 \\
 & =  \int_0^t \int_{U_0\times \mathbb{R}_+^0 } (|\mathcal{Z}_s^{j, \star} |^2- |\mathcal{Z}_s|^2) 1_{[0,\; \s_j(\mathcal{Z}_s, v)f(s,\mathcal{X}_s |v)]}(r)  m(s,v)dv Q(d\theta)d\phi dr ds  +\mathcal{M}_t^j\\ 
 & = |\mathcal{Z}_0|^2+  \int_0^t \int_{U_0} (|\mathcal{Z}_s^{j, \star} |^2- |\mathcal{Z}_s|^2) \s_j(\mathcal{Z}_s, v)f(s,\mathcal{X}_s, v) dv Q(d\theta)d\phi ds +\mathcal{M}_t^j  \\& =   |\mathcal{Z}_0|^2 + \int_0^t \int_{U_0} (|v|^2-|v^{j, \star} |^2+E_j(\mathcal{Z}_s,v)) \s_j(\mathcal{Z}_s, v)f(s,\mathcal{X}_s, v) dv Q(d\theta)d\phi  ds   +\mathcal{M}_t^j
 \end{eqnarray*}
 where $\mathcal{M}_\cdot^j$ is a martingale and the last inequality is due to \eqref{notmoresymm}. It follows 
 \begin{eqnarray*}
 &|\mathcal{Z}_t|^2 - |\mathcal{Z}_0|^2\\
 &= \frac{1}{2} \int_0^t \int_{U_0} (|\mathcal{Z}_s^{j, \star} |^2- |\mathcal{Z}_s|^2) \s_j(\mathcal{Z}_s, v) f(s,X_s, v) dv Q(d\theta)d\phi ds \\&  + \frac{1}{2} \int_0^t \int_{U_0} (|v|^2-|v^{j, \star} |^2 +E_j(\mathcal{Z}_s,v) ) )\s_j(\mathcal{Z}_s, v) f(s,\mathcal{X}_s, v) dv Q(d\theta)d\phi ds   +\mathcal{M}_t^j
 \\& = \frac{1}{2}\int_0^t \int_{U_0} \{2(\mathcal{Z}_s+v, \alpha_j(\mathcal{Z}_s,v,\xi) )+E_j(\mathcal{Z}_s,v) )\} \s_j(\mathcal{Z}_s, v) f(s,\mathcal{X}_s, v) dv Q(d\theta)d\phi ds  +\mathcal{M}_t^j
 \\& =  |\mathcal{Z}_0|^2+ \frac{1}{2}\int_0^t \int_{U_0} \{2(\mathcal{Z}_s+v, \alpha_j(\mathcal{Z}_s,v,\xi) )+\frac{2d(\mathcal{Z}_s,B_j)}{1+d(\mathcal{Z}_s,B_j)}(\mathcal{Z}_s, \alpha_j(\mathcal{Z}_s,v,\theta, \phi))\}\\& \times \s_j(\mathcal{Z}_s, v) f(s,\mathcal{X}_s, v) dv Q(d\theta)d\phi ds  +\mathcal{M}_t^j
 \\&=  |\mathcal{Z}_0|^2+ \int_0^t \int_{U_0}\{ (\mathcal{Z}_s+v, v -\frac{\mathcal{Z}_s}{1+ d(\mathcal{Z}_s,B_j)}) +\frac{d(\mathcal{Z}_s,B_j)}{1+d(\mathcal{Z}_s,B_j)}(\mathcal{Z}_s,  v- \frac{\mathcal{Z}_s}{1+d(\mathcal{Z}_s,B_j)} )\}\\& \times \sin^2(\frac{\theta}{2})\s_j(\mathcal{Z}_s, v) f(s,\mathcal{X}_s, v) )dv Q(d\theta)d\phi ds  +\mathcal{M}_t^j
 \end{eqnarray*}
 where we used \eqref{PrameterTrans}  and \eqref{symmGamma} in the last equality.\\
Recalling the definition of $\sigma_j$ in \eqref{sigma-loc}, we get   the following inequality.
 \begin{eqnarray*}  &\mathbb{E}[|\mathcal{Z}_t|^2] - \mathbb{E}[ |\mathcal{Z}_0|^2]\\
 &\leq  c \mathbb{E}\left[\int_0^t \int_{U_0} |\frac{\mathcal{Z}_s}{1+d(z,B_j)} -v|
 \left( |v|^2 +2 d(z,B_j) (v, \frac{\mathcal{Z}_s}{1+ d(z,B_j)})\right) \sin^2(\frac{\theta}{2})Q(d\theta)d\phi f(s,\mathcal{X}_s,v)dvds \right]\end{eqnarray*}
\begin{eqnarray*}
&\leq 
   \hat{K}\mathbb{E}\left[ \int_0^t \int_{U_0} (|\frac{\mathcal{Z}_s}{1+d(z,B_j)}|
+|v|)
\left( |v|^2 +2 d(z,B_j) |v| |\frac{\mathcal{Z}_s}{1+ d(z,B_j)}|\right)   \sin^2(\frac{\theta}{2}) Q(d\theta)d\phi f(s,\mathcal{X}_s,v)dvds \right]\\ &\leq 
\mathbb{E}[ |\mathcal{Z}_0|^2] +  \hat{K} \mathbb{E}\left[\int_0^t \int_{U_0} (|\frac{\mathcal{Z}_s}{1+d(z,B_j)}|
+|v|
) |v|^2 \sin^2(\frac{\theta}{2}) Q(d\theta)d\phi f(s,\mathcal{X}_s,v)dvds\right] \\& + 2 \hat{K} \mathbb{E}\left[\int_0^t \int_{U_0} (|\frac{\mathcal{Z}_s}{1+d(z,B_j)}|
 +|v|
) |v| |\mathcal{Z}_s\sin^2(\frac{\theta}{2}) Q(d\theta)d\phi f(s,\mathcal{X}_s,v)dvds\right]\\
 &\leq 
 3 \hat{K} \mathbb{E}\left[\int_0^t \int_{U_0} (|\mathcal{Z}_s|^2+1 +|v|
) (|v|^2 + |v|) \sin^2(\frac{\theta}{2}) Q(d\theta)d\phi f(s,\mathcal{X}_s,v)dvds\right]
\\ &\leq  
  (3\hat{K}C_T + \hat{K} c_T) \int_0^T  (\sup_{[0,t]}\mathbb{E}[|\mathcal{Z}_{s}|^2]+1)dt\int_0^\pi\int_0^{2\pi}\sin^2(\frac{\theta}{2}) Q(d\theta)d\phi \\
 &\leq  K_T \int_0^T  (\sup_{[0,t]}\mathbb{E}[|\mathcal{Z}_{s}|^2]+1)dt+\hat{K} c_T \int_0^\pi\int_0^{2\pi} \sin^2(\frac{\theta}{2}) Q(d\theta)d\phi
 \end{eqnarray*}
 where 
 the constants  $C_T$ and  $c_t$ have been   introduced in condition {\bf B4} and resp. {\bf B6} and  $K_T := 3\hat{K}C_T + \hat{K} c_T$. we  have repeatedly used  the inequality $|x|^\epsilon\leq |x|^2 +1 \quad \forall x \in \mathbb{R}^d \quad \text{for} \quad \epsilon\in (0,2]$.
   Here $\hat{K}$ and $K_T$ are positive finite constants.\\

 {\it Proof of \eqref{boundStbyL2j}:}\\
  Let us  consider the stochastic integral

\begin{equation}\label{stochintmod}
\hat{I}_j(\mathcal{Z})_{t}:= \int_0^t  \int_{U_0\times \R^+_0} |\alpha_j(\mathcal{Z}_{s_-}, v, \theta, \phi)|1_{[0,\; \s_j(\mathcal{Z}_{s_-}, v)f(s, \mathcal{X}_s | v)]}(r)
 d{N}, 
\end{equation}

Using \eqref{growth},  {\bf{B4}}, and {\bf{B6}}
 it follows 
\begin{align}
&\mathbb{E}[\hat{I}_j(\mathcal{Z})_{T}] =  
\int_0^T \int_{U_0 \times \R^+_0} \mathbb{E}[  |\alpha_j(\mathcal{Z}_{s}, v, \theta, \phi)| \s_j(\mathcal{Z}_{s_{-}}, v)f(s, \mathcal{X}_s, v)]dv d\phi Q(d\theta)ds \notag
\\& \leq C 
  \left( \int_0^T \mathbb{E} [ |\frac{\mathcal{Z}_s}{1+d(z,B_j)}|^2
\int_{\mathbb{R}^3} f(s,\mathcal{X}_s,v)dv] ds + \int_0^T \int_{\mathbb{R}^3} \mathbb{E}[|v|^2
f(s, \mathcal{X}_s, v)]dvds\right) \notag
\\& \leq C 
  \left(\int_0^T  \sup_{s'\in [0,s]} \mathbb{E} [   |\mathcal{Z}_{s'}|^2
] \int_{\mathbb{R}^3} f(s,\mathcal{X}_s,v)dv]  ds  + \int_0^T \int_{\mathbb{R}^3} \mathbb{E}[|v|^2
f(s, \mathcal{X}_s, v)]dvds\right) \notag
 \\& \leq C_T C
    \left(\int_0^T \sup_{s\in [0,t]} \mathbb{E} [ |{Z}_s|^{2}] dt +  \int_0^T \sup_{x \in \mathbb{R}^3}\{\int_{\mathbb{R}^3}( |v|^2
+1)f(s, x, v)dv\}ds \right)\notag\\&
   \leq C_T C
   \left(\int_0^T \sup_{s\in [0,t]} \mathbb{E} [ |{Z}_s|^{2}] dt +T \mathcal{C}_T \right) \label{ineqIsj}
\end{align} 
   
where we have chosen $ \mathcal{C}_T $ to be a big enough but finite constant. 
 \eqref{boundStbyL2j} follows from \eqref{expsupSj} and \eqref{ineqIsj}.

\end{proof}
We are now in a position to prove Theorem \ref{Theorem MR}.\\

\noindent {\bf{ Proof of Theorem \ref{Theorem MR}}:}
\begin{proof}
Let us  first remark that the following equalities hold
\begin{equation}\label{alphaeqalphaj}
\alpha_j(z,v,\theta, \phi)=\alpha(z,v,\theta, \phi), \quad 
\sigma_j(z,v)=\sigma(z,v) \quad \forall z \in \mathbb{R}^3 :\, |z|\leq j
\end{equation}

Let $(X_\cdot^{j}, Z_\cdot^{j})$$ \in S_T^1$ be the  solution of  \eqref{eq-velr-j-cor}, \eqref{eq-spacer-j-cor} with finite initial second moment. 
Let \[
\tau_{_{1}} = \inf\{t \in [0, T]: |Z_t^1| > 1 \}.\] 
We denote the solution $(X_\cdot^1, Z_\cdot^1)$ as $(X^{(1)}, Z^{(1)})$ in the time interval $[0, \tau_1)$. Its law is 
a probability measure $P_1$ on the canonical path space $(\Omega, \mathcal{F}_{\tau_{_1}})$. 
Next, for each fixed $\omega \in \O$, consider the equations \eqref{eq-spacer-j-cor} and \eqref{eq-velr-j-cor}, with initial condition given by $(X^1_{\tau_{_1}(\omega)}, Z^1_{\tau_{_1}(\omega)})$ with initial time being ${\tau_{_1}(\omega)}$. The above procedure 
can be used to obtain a solution $(X^{(2, \omega)}, Z^{(2, \omega)})$ until time $\tau_{_{2, \omega}}$ where 
\[
\tau_{_{2, \omega}} = \inf\{t \in [\tau_{_{1}}(\omega), T]: |Z^{(2, \omega)}(t)| > 2\}.\] 
The law of $(X^{(2, \omega)}, Z^{(2, \omega)})$ until $\tau_{_{2, \omega}}$ 
is denoted by $Q_{\omega}$, a probability measure for each $\omega \in \O$. 
It gives us a regular conditional probability distribution (rcpd)  given $\mathcal{F}_{\tau_{_1}}$.  Hence, by a theorem of Stroock and Varadhan (Lemma 3.6 in \cite{SV1}, cf. Ch. 12 \cite{SV}, and  \cite{St}) , we patch the measures $P_1$ and $Q$ so that there 
exists a unique probability measure $P_2$ on $(\Omega, \mathcal{F}_{\tau_{_2}})$ such that 
$P_2 = P_1$  on $\mathcal{F}_{\tau_{_1}}$, and the rcpd of $P_2$ given 
$\mathcal{F}_{\tau_{_1}}$ is $Q$. 

By  \eqref{alphaeqalphaj} and the definitions of $P_1$ and $Q$, we obtain that 
$P_2$ solves the martingale problem posed by  \eqref{eq-spacer} and \eqref{eq-velr} up to time $\tau_{_2}$. 
The procedure is repeated to obtain a solution on $[0, \tau)$ where $\tau = \lim_{n \to \infty} \tau_n$. 


 We will use the statement in Theorem \ref{Th moment control} to prove that
\begin{equation} \label{fulltime}
\mathrm{P}( \cup_{j\in \mathbb{N}}\{\tau_j=T\})=1
\end{equation} 
so that $\tau = T$ a.s.
Since  $\mathrm{P}(\tau_j\leq \tau_{j+1})=1$ $\forall j\in \mathbb{N}$,  it would then follow that $
\displaystyle{\lim_{n \to \infty}  \tau_n = T} $ a.s. Hence a solution of \eqref{eq-spacer} and \eqref{eq-velr} exists for $t \in [0, T]$ and the proof would be over.\\


{\bf Proof of \eqref{fulltime}:}
     \begin{align*} \mathrm{P}(\tau_j<T)&=\mathrm{P}(\sup_{t\in [0,T]} |Z^{j}_t|>j)\\& 
     \leq \frac{\mathbb{E}[ \sup_{t\in [0,T]} |Z^{j}_t| ]}{j}=\frac{\|\mathcal{Z}_\cdot\|_{S^{1}_T}} {j}\\&
     \leq 
    \frac{ T k_T (M_T  + \mathbb{E}[|\mathcal{Z}_0|^2]) e^{T K_T}  + m_T + \mathbb{E}[|\mathcal{Z}_0|]}{j}
     \end{align*}
     by the estimate \eqref{growthinST1}. It follows  that
     \begin{equation} \label{fulltimecomplement}
     \mathrm{P}( \cap_{j\in \mathbb{N}}\{\tau_j<T\})=\lim_{j\to\infty} \mathrm{P}(\{\tau_j< T)\}=0.
     \end{equation}   
Hence, \eqref{fulltime} is proven.\\
The finiteness of second moments expressed in equation \eqref{secondmomentmeasure}  in the statement of Theorem \ref{Theorem MR} is a consequence of the estimate \eqref{GrownwallZj} in Theorem \ref{Th moment control}  
\end{proof}

\subsection{Construction of  Boltzmann processes with densities satisfying \eqref{eqn1.1}} \label{Par Construction BP}

In this section we assume that  $\{f(t,x,z)\}_{t\in [0,T]}$ is  a  collection of densities which solves the Boltzmann equation 
\eqref{eqn1.1} and   satisfies 
hypothesis {\bf {B0}}- {\bf {B6}}. 


\begin{theorem} \label{thm-existence-given-f} 
 Let    $(X_t,Z_t)_{t\in [0,T]}$ be  a stochastic process solving  (\eqref{eq-spacer}, \eqref{eq-velr}).  
 Suppose that  $(X_t,Z_t)_{t\in [0,T]}$ has  time marginals with  density $g(t,x,z)$, for each $t\in [0,T]$ which satisfy  hypothesis {\bf {B0}} - {\bf{B3}} for $\gamma=1$. Suppose   $\{f(t,x,z)\}_{t\in [0,T]}$ and  $\{g(t,x,z)\}_{t\in [0,T]}$ satisfy condition 
 {\bf C1} stated in Section \ref{Par bilinear BE}.  Let  $g(0,x,z)$ $=$$ f(0,x,z)$ a.s..    
Then $g(t,x,z)=$ $f(t,x,z)$ $\quad a.s.$ for all $t\in [0,T]$.  
 \end{theorem}

 \begin{proof}  The statement in Theorem \ref{thm-existence-given-f}  is proven, once we show  the following equality 
  for  the relative entropy 
 \begin{equation} \label{relative entropy eq}
 R_t(g|f):= \int_{\mathbb{R}^6} \ln{\left(\frac{g(t.x.z)}{f(t,x,z)}\right)} g(t,x,z) dxdz = 0 \quad \forall t\in [0,T]
 \end{equation}
\\ {\it Proof of equality \eqref{relative entropy eq}:}\\
We first apply the It\^o formula (Theorem 5.1 in Chapter II, Section 5  \cite{IW}) to $\ln({g(t, X_t, Z_t)})$ , where   $(X_t,Z_t)_{t\in [0,T]}$  solves  (\eqref{eq-spacer}, \eqref{eq-velr}) and take its expectation.  By similar arguments as in the proof of Theorem \ref{Ito} we obtain

 \begin{align}\label{lngIto}
  & \int_{\mathbb{R}^6} \ln{(g(t,x,z))} g(t,x,z) dxdz -
  \int_{\mathbb{R}^6} \ln{(g(0,x,z))}  g(0,x,z) dxdz \notag \\& 
  = \int_0^t  \int_{\mathbb{R}^6 } \left(\frac{\partial}{ \partial s} \ln{(g(s,x,z)}\right) g(s,x,z) dx dz ds \notag \\& +\int_0^t  
 \int_{\mathbb{R}^6 }  (z \cdot \nabla_x \ln{(g(s,x,z))} g(s,z,x) dxdzds\notag\\&
 +  \int_0^t  \int_{\mathbb{R}^9 \times \Xi} \{  \ln{(g(s,x,z^\star)}- \ln{(g(s,x,z)}\}\notag \\&\times \sigma(|z-v|) f(s,x,v)g(s,x,z)Q(d\theta)d\phi dv  dxdzds 
 \end{align}
 By the same arguments as in the Proof of Theorem \ref{thm-uniqueness IWlinBE-given-f}, the first and second  term in the right side of \eqref{lngIto} vanish.
 \eqref{lngIto} simplifies then to 
 \begin{align}\label{lngIto2}
  & \int_{\mathbb{R}^6} \ln{(g(t,x,z))} g(t,x,z) dxdz -
  \int_{\mathbb{R}^6} \ln{(g(0,x,z))}  g(0,x,z) dxdz \notag \\&
 = \int_0^t  \int_{\mathbb{R}^6 \times \mathbb{R}^3 \times \Xi} \{  \ln{(g(s,x,z^\star))}- \ln{(g(s,x,z))}\}\notag \\&\times \sigma(|z-v|) f(s,x,v)g(s,x,z)Q(d\theta)d\phi dv  dxdzds 
 \end{align}
  
 We now apply the It\^o formula to $\ln({f(t, X_t, Z_t)})$, where   $(X_t,Z_t)_{t\in [0,T]}$  solves ( \eqref{eq-spacer}, \eqref{eq-velr}).  Similar to \eqref{eqIto} in the proof of Theorem \ref{Ito} we obtain 
  \begin{align}
  &\ln{(f(t,X_t, Z_t)} -\ln{(f(0,X_t, Z_t)}\notag \\&=\int_0^t \frac{1}{f(s,X_s, Z_s)} \frac{\partial}{ \partial s} f(s,X_s,Z_s) ds +\int_0^t Z_s \cdot \frac{\nabla_x {(f(s,X_s,Z_s))}}{f(s,X_s,Z_s)} ds \notag \\& +  \int_0^{t} \int_{ U_0} \{\ln (f(s,X_s, Z_s + \alpha(Z_s,v,\theta, \phi))) -\ln(f(s,X_s, Z_s))  \} \notag\\& \times \s(|Z_{s} - v|)f(s,X_s,v) dv Q(d\theta) d\phi ds  \notag \\
 & + M^{t}_0(\ln) \notag 
 \end{align} where $ M^{t}_0(\ln)$ is a martingale. By taking the expectation and considering  that  $\{f(t,x,z)\}_{t\in [0,T]}$ is  a  collection of densities which solves the Boltzmann equation 
\eqref{eqn1.1} we obtain 
\begin{align}\label{lnfIto2}
  & \int_{\mathbb{R}^6} \ln{(f(t,x,z))} g(t,x,z) dxdz -
  \int_{\mathbb{R}^6} \ln{(f(0,x,z))}  g(0,x,z) dxdz \notag \\&
 = \int_0^t  \int_{\mathbb{R}^6 } \frac{Q(f, f) (s,x,z)} {f(s,x,z)} g(s,x,z) dxdzds \notag \\& + \int_0^t  \int_{\mathbb{R}^6 \times \mathbb{R}^3 \times \Xi} \{  \ln{(f(s,x,z^\star))}- \ln{(f(s,x,z))}\}\notag \\&\times \sigma(|z-v|) f(s,x,v)g(s,x,z)Q(d\theta)d\phi dv  dxdzds 
 \end{align}

The rest of the proof is identical to the proof of Theorem \ref{thm-uniqueness IWlinBE-given-f} , since \eqref{lngIto2} is identical to  \eqref{lng} and \eqref{lnfIto2} is identical to \eqref{InfIto}. 

 \end {proof}
 In Theorem \ref{Theorem MR} we proved the existence of a weak solution to the SDE (\eqref{eq-spacer}, \eqref{eq-velr}), given  $\{f(t,x,z)\}_{t\in [0,T]}$.
  From Theorem \ref{thm-existence-given-f}, it follows under the assumption that $\{f(t,x,u)\}_{t\in [0,T]}$ solves the  Boltzmann equation \eqref{eqn1.1}, that  the stochastic process  $(X_t,Z_t)_{t\in [0,T]}$ in Theorem   \ref{Theorem MR}  
  solves the McKean -Vlasov equation associated 
 to the Boltzmann equation \eqref{eqn1.1} and is a Boltzmann process  with densities $\{f(t,x,z\}_{t \in [0,T]}$ up to time  $T$. This is stated in the final result  below:
\begin{theorem}\label{mainThexistenceBoltzmann} 
  Assume that $\{f(t,x,u)\}_{t\in [0,T]}$ is  a  collection of densities which solves the  Boltzmann equation \eqref{eqn1.1}, 
 and satisfies the hypotheses ${\bf {B}_0} - {\bf{B}_6}$.
 Suppose that  the weak  solution $(X_t, Z_t)_{t\in [0,T]}$  of the stochastic system (\eqref{eq-spacer}, \eqref{eq-velr}) has its distribution that admits a probability density at each time $t \in [0, T]$ given by $g(t,x,u)$. If condition {\bf{C1}} is satisfied by $\{f(t,x,u)\}_{t\in [0,T]}$ and    $\{g(t,x,u)\}_{t\in [0,T]}$ and  $(X_0, Z_0)$ has  finite second moment, then  (\eqref{eq-spacer} , \eqref{eq-velr})  is a stochastic  McKean-Vlasov equation  and coincides with 
 (\eqref{eq-spaceB},\eqref{eq-velB}).  Its   solution 
$(X_t, Z_t)_{t\in [0,T]}$  has  values in $\mathbb{D}\times \mathbb{D}$ and  
  its Fokker-Planck  equation is the Boltzmann equation \eqref{eqn1.1}. Hence  $(X_\cdot, Z_\cdot)$ is a Boltzmann process.
  \end{theorem}

 \begin{proof}
  The result follows from Theorem \ref{Theorem MR} and Theorem \ref{thm-existence-given-f}.
 \end{proof}

\medskip {\bf Data availability statement: } This manuscript has no associated data.

 \end{document}